\documentclass [a4paper,12pt]{article}
\usepackage{amsmath}%用于Align环境
\usepackage{amsfonts}%用于\mathbb{}
\usepackage{indentfirst}
\usepackage{esint}
\usepackage{latexsym}
\usepackage{bbm}
\usepackage{amsfonts}
\usepackage{mathrsfs}
\usepackage{amssymb}
\usepackage{amsmath}
\usepackage{amsthm}
\usepackage{latexsym}
\usepackage{indentfirst}
\usepackage{enumitem}
\usepackage{bm}
\usepackage{esint}
\usepackage{hyperref}
\usepackage{cleveref}
\usepackage{cite}
\usepackage[numbers,square]{natbib}
\numberwithin{equation}{section}
\usepackage{color}
\allowdisplaybreaks%用于自动断页
\usepackage{amsthm}%用于定理环境
\usepackage{amssymb}%第一次用于写右箭头\nrightarrow命令
\usepackage{enumerate}%计时器，用于编号公式之类;还用于列表环境。
\usepackage{ulem}%用于划删除线
\allowdisplaybreaks

\newtheorem{theorem}{Theorem}[section]
\newtheorem{lemma}[theorem]{Lemma}
\newtheorem{thm}[theorem]{Theorem}

\newtheorem{rmk}[theorem]{Remark}

\makeatletter%以下四行是用于写大小写罗马数字的，本文暂时只用得到大写。
\usepackage{amsmath}%用于Align环境
\usepackage{amsfonts}%用于\mathbb{}
\usepackage{indentfirst}
\usepackage{esint}
\usepackage{latexsym}
\usepackage{authblk}
\usepackage{color}
\UseRawInputEncoding
\usepackage{amsthm}%用于定理环境
\usepackage{amssymb}%第一次用于写右箭头\nrightarrow命令
\usepackage{enumerate}%计时器，用于编号公式之类;还用于列表环境。
\usepackage{ulem}%用于划删除线
\makeatletter%以下四行是用于写大小写罗马数字的，本文暂时只用得到大写。

\newcommand{\Rmnum}[1]{\expandafter\@slowromancap\romannumeral #1@}
\makeatother
\begin{document}
\title{Homogenization and Convergence Rates for Periodic Parabolic Equations with Highly Oscillating Potentials}
\author{Yiping Zhang\thanks{Yanqi Lake Beijing Institute of Mathematical Sciences and Applications, Beijing 101408, China and  Yau Mathematical Sciences Center, Tsinghua University, Beijing 100084, China, (zhangyiping161@mails.ucas.ac.cn).}}
\date{}
\maketitle
\begin{abstract}
This paper considers a family of second-order periodic parabolic equations with highly
oscillating potentials, which have been considered many times for the time-varying potentials in stochastic homogenization.
Following a standard two-scale expansions illusion, we can guess and succeed in determining the homogenized equation in different cases that the potentials satisfy the corresponding assumptions, based on suitable uniform estimates of the $L^2(0,T;H^1(\Omega))$-norm for the solutions. To handle the more singular case and obtain the convergence rates in $L^\infty(0,T;L^2(\Omega))$, we need to estimate the Hessian term as well as the t-derivative term more exactly, which may be depend on $\varepsilon$. The difficulty is to find suitable uniform estimates for the $L^2(0,T;H^1(\Omega))$-norm and suitable estimates for the higher order derivative terms.

\end{abstract}

\section{Introduction}
Let $\Omega\subset\mathbb{R}^d$ be a bounded $C^{1,1}$ domain with $d\geq 2$ and $0<T<\infty$. For $0<\varepsilon<1$, we investigate the following parabolic equations in homogenization theory:
\begin{equation}\left\{\begin{aligned}
\partial_t u_\varepsilon-\Delta u_\varepsilon -\varepsilon^{-\gamma}W\left(\frac{x}{\varepsilon^{k_1}},\frac t {\varepsilon^{k_2}}\right)u_\varepsilon&=f, \quad\quad\text{in }\Omega\times (0,T),\\
u_\varepsilon&=0, \quad\quad\text{on }\partial\Omega\times (0,T),\\
u_\varepsilon&=g, \quad\quad\text{on }\Omega\times \{t=0\},
\end{aligned}\right.\end{equation}
where $W(y,\tau)\in L^\infty(\mathbb{T}^{d+1})$  is 1-periodic in $(y,\tau)$ with $\int_{\mathbb{T}^{d+1}}W(y,\tau)dyd\tau=0$, $\mathbb{T}\cong \mathbb{R}/\mathbb{Z}$ denoting the unit torus and $\gamma$ being a positive constant depending only on $k_1$ and $k_2$. Under suitable conditions on $f$ and $g$, and according to the classical parabolic theory, it is known that, without the periodicity condition on $W$, there exists a unique weak solution $u_\varepsilon$ in $L^2(0,T;H_0^1(\Omega))$ to the equation $(1.1)$ with the $L^2(0,T;H_0^1(\Omega))$-norm depending on $\varepsilon$. Then the first difficulty is to derive uniform estimates of the $L^2(0,T;H_0^1(\Omega))$-norm for the solution $u_\varepsilon$ given by the classical parabolic theory. Moreover, in order to guarantee the strong convergence in $L^2(\Omega\times(0,T))$, we need also to obtain the uniform estimates of $||\partial_t u_\varepsilon||_{L^2(0,T;H^{-1}(\Omega))}$.

Our goal is to determine the limiting equation and obtain the  convergence rates, as $\varepsilon\rightarrow 0$ with suitable potentials $\varepsilon^{-\gamma}W\left({x}/{\varepsilon^{k_1}}, t / {\varepsilon^{k_2}}\right)$, where a suitable $\gamma$ is the correct value for $\gamma$ such that the limiting effect of the highly oscillating term is nontrivial. The homogenization results (without the results of convergence rates) have been widely considered in stochastic homogenization.

For the one-dimensional case with $\Omega=\mathbb{R}$, $f=0$, $g\in L^2(\mathbb{R})\cap C_b(\mathbb{R})$, and the potential W being a stationary random field defined on a probability space with zero-expectation, Pardoux and Piatnitski \cite{MR2962093} have obtained the homogenized equation with nontrivial potentials. In this case, the solution $u_\varepsilon$ has an explicit expression and is given by the Feynman-Kac formula:
$$\begin{aligned}u_\varepsilon(t,x)=\mathbb{E}\left[g(x+B_t)\exp\left(\varepsilon^{-\gamma}\int_0^tW
\left(\frac{s}{\varepsilon^{k_2}},\frac{x+B_s}{\varepsilon^{k_1}}\right)ds\right)\right].
\end{aligned}$$
 Precisely, they have considered three cases: Case (1), $k_2=0$, $k_1>0$, $\gamma=k_1/2$; Case (2), $k_1=0$, $k_2>0$, $\gamma=k_2/2$; Case (3), $0<2k_1\leq k_2$, $\gamma=\max(k_2/4+k_1/2,k_2/2)$. They succeeded in determining the limit equation for these three cases under suitable assumptions satisfied by the potential $W$. See also the previous works \cite{MR2233186,MR2451056,MR2718269} considering this topic.

Later on, under suitable scaling of the
amplitude of the potential $W$, Hairer and Pardoux \cite{MR3327517} proved the convergence to a deterministic heat equation
with constant potential for the case $\gamma=\max(k_2/4+k_1/2,k_2/2)$ with $k_2>2 k_1$, thus completing the results previously obtained in Pardoux and Piatnitski \cite{MR2962093}. Precisely, the assumptions are that $W$ is stationary, centered, continuous and $C^1$ in the $x$-variable with $E(|W|^p+|\partial_x W|^p)<\infty$ for any $p>0$, and the maximal correlation function $\varrho$ defined in this paper satisfies $\varrho(R)\leq C(1+R)^{-q}$ for any $q>0$.

For the multidimensional case ($d\geq 3$),  Gu and   Bal \cite{MR3274693} have considered the equation $\partial_t u_\varepsilon-\Delta u_\varepsilon-i\varepsilon^{-\gamma}W(x/\varepsilon,t/\varepsilon^k)u_\varepsilon=0$, and they succeeded in determining the limit equation for $\gamma=\max(k/2,1)$ with $\infty>k>0$ under the assumption that $W$ satisfies an extra finite range of dependence. Moreover, then have also obtained the limit equation convergence to SPDE if $k=\infty$.

For the error estimates,  Gu and  Bal \cite{MR3449303} have considered the equation $\partial_t u_\varepsilon(x,t,\omega)-\Delta u_\varepsilon(x,t,\omega)-i\varepsilon^{-1}W(x/\varepsilon,\omega)u_\varepsilon(x,t,w)=0$. Under suitable conditions for the potential $W$, they can not only determine the effective equation, but also obtain the following error estimates for long-range-correlated Gaussian potentials:
\begin{equation*}
\mathbb{E}\left\{\left|u_{\varepsilon}(t, x)-u_{\mathrm{hom}}(t, x)\right|\right\} \leq C (1+t) \begin{cases}\sqrt{\varepsilon} & d=3, \\ \varepsilon \sqrt{|\log \varepsilon|} & d=4, \\ \varepsilon & d>4.\end{cases}
\end{equation*}

In fact, there are many other works considering this topic, such as \cite{MR3274693,MR3312591}. We refer to a survey \cite{MR3318383} on this topic and the reference therein for more related results for this topic on stochastic case.

Now, back to our periodic setting, we can indeed handle some more singular potential terms, such as the case $\gamma=k-1$ for $k_1=1$ and $3\geq k_2=k>2$ (note that $k-1>k/2$  for $k\in (2,3]$), under
suitable assumptions for the $W$. Moreover, we can also obtain the convergence rates provided by the illusion of two-scale expansions with time-varying potentials $W$. In fact, in order to determine the effective equation and obtain the convergence rates, there are two difficulties that we need to overcome. The
first difficulty, as pointed out early, is to derive a uniform $L^2(0,T;H_0^1(\Omega))$-norm for the solution
$u_\varepsilon$ and to bound $||\partial_t u_\varepsilon||_{L^2(0,T;H^{-1}(\Omega))}$, so as to guarantee a
weak convergence in $L^2(0,T;H_0^1(\Omega))$ and a strong convergence in $L^2(\Omega\times (0,T))$, which are
usually enough to determine the homogenized equation. However, the above uniform estimates are not enough to
bound the convergence rates, and we need the estimates of the higher order derivative terms of $u_\varepsilon$ to bound these
error terms. Then the second difficulty is to derive  a more exact estimate on $\nabla^2 u_\varepsilon$ and
$\partial_t u_\varepsilon$, which may be depend on $\varepsilon$, since at a first glance, we can easily have $||\partial_t u_\varepsilon||_{L^2(\Omega_T)}+||\nabla^2 u_\varepsilon||_{L^2(\Omega_T)}\leq C \varepsilon^{-\gamma}$, which is not enough to determine the limiting equation and bound these error terms.

For the theory of periodic homogenization,  there is a vast and rich mathematical literature over the last forty years. Most of these works are focused on qualitative results, such as proving the existence of a homogenized equation. However, until recently, nearly all of the quantitative theory, such as the convergence rates in $L^2$ and $H^1$, the $W^{1,p}$-estimates, the Lipschitz estimates, the asymptotic expansions of the Green functions and fundamental solutions, were confined in periodic homogenization. There are many good expositions on this topic, see for instance the books \cite{MR2839402,MR1329546,MR3838419}
for periodic case, see also the book \cite{MR3932093} for the stochastic case.

Moreover, for the periodic or almost periodic case on this topic, we refer to \cite[Chapter 2.4.3]{MR2839402} and \cite{MR430548} for a homogenization result under some special cases.\\

Next, we briefly outline the structure of this paper.

In the next section, we give a heuristic illusion via the  two-scale expansions. That some conditions will be necessary  so as to guarantee these processes, and we can handle some more singular cases will be all included in these processes, even if we only calculate some cases that are easy to perform.

In Section 3, we gather the information obtained in Section 2, and introduce the main results in this paper, including the homogenization results and the convergence rates.

In Sections 4-6, we give the proofs of our main results given in Section 3.\\

At the end of this section, we introduce some notations that will be used frequently.\\

$\Omega$ denotes a bounded $C^{1,1}$ domain in $\mathbb{R}^d$ with $d\geq 2$. $\Omega_T=:\Omega\times (0,T)\subset \mathbb{R}^{d+1}$ for some positive constant $0<T<\infty$.

$\mathbb{T}\cong \mathbb{R}/\mathbb{Z}$ denotes the unit torus. We say $h(y,\tau)$ is 1-periodic in $(y,\tau)$, which means that for any $Z=(z',z_{d+1})\in \mathbb{Z}^{d+1}$, then $h(y+z',\tau+z_{d+1})=h(y,\tau)$. And for simplicity, we just say $h(y,\tau)$ is 1-periodic. $h\in \mathcal{B}(\mathbb{T}^{d+1})$ means that $h(y,\tau)$ is 1-periodic in $(y,\tau)$ with the norm $||h||_{\mathcal{B}(\mathbb{T}^{d+1})}$.

 For a 1-periodic function $h(y,\tau)$, we denote
\begin{equation*}\mathcal{M}(h)=:\int_{\mathbb{T}^{d+1}}h(y,\tau)dyd\tau,\ \mathcal{M}_y(h)(\tau)=:\int_{\mathbb{T}^{d}}h(y,\tau)dy,\ \mathcal{M}_\tau(h)(y)=:\int_{\mathbb{T} }h(y,\tau)d\tau\end{equation*}
and
we denote $h^\varepsilon(x,t)=:h(x/\varepsilon^{k_1},t/\varepsilon^{k_2})$ if the content is understood. Note that, in the following sections, we, for simplicity, always assume that $k_1=1$ and $k_2=k$, then $h^\varepsilon(x,t)=:h(x/\varepsilon,t/\varepsilon^{k})$ for different choices of $k$. Moreover, if $h=h(x)$, then $h^\varepsilon(x)=:h(x/\varepsilon)$ and if $h=h(t)$, then $h^\varepsilon(t):=h(t/\varepsilon^k)$.

We use the following H\"{o}lder spaces: $$\begin{aligned}&[u]_{C^{\alpha, \alpha / 2}(Q)} =\sup _{\substack{(x, t),(y, s) \in Q \\
(x, t) \neq(y, s)}} \frac{|u(x, t)-u(y, s)|}{\left(|x-y|+|t-s|^{1 / 2}\right)^{\alpha}},||u||_{C^{\alpha, \alpha / 2}(Q)}=[u]_{C^{\alpha, \alpha / 2}(Q)}+||u||_{L^\infty(Q)},\\
&||u||_{C^{1+\alpha, (1+\alpha) / 2}(Q)}=||\nabla u||_{C^{\alpha, \alpha / 2}(Q)}+||u||_{C^{\alpha, \alpha / 2}(Q)},
\ 0<\alpha < 1.\end{aligned}$$

For simplicity, we use the following notations: $L^\infty_TL^2_x=:L^\infty(0,T;L^2(\Omega))$ and $L^2_TH^1_x=:L^2(0,T;H^1(\Omega))$.

Moreover, denote $v_\varepsilon=:e^{-Nt}u_\varepsilon$ for some positive constant $N$, varying in different cases but only depending on $W$ and $d$, then it is easy to check that $v_\varepsilon$ satisfies the following equation:
\begin{equation}\left\{\begin{aligned}
\partial_t v_\varepsilon-\Delta v_\varepsilon+\left(N-\varepsilon^{-\gamma}W^\varepsilon\right)v_\varepsilon&=\tilde{f}, \quad\quad\text{in }\Omega\times (0,T),\\
v_\varepsilon&=0, \quad\quad\text{on }\partial\Omega\times (0,T),\\
v_\varepsilon&=g, \quad\quad\text{on }\Omega\times \{t=0\},
\end{aligned}\right.\end{equation}
where $\tilde{f}=fe^{-Nt}$. Similarly, we denote $v_0=:e^{-Nt}u_0$, where $u_0$ is the solution to the homogenized equation.

\section{Heuristic Illusion: Two-Scale Expansions}
In this section, we pay attention to the heuristic illusion so as to guess the effective equation by using the two-scale expansions and to find the suitable assumptions satisfied by the potential $W$. Note that the correct value for $\gamma$ will be given in the next section.

 In fact, there are two basic principles hidden behind in these processes: the one is that the solution $u_0$ to the effective equation should not depend on $(y,\tau)$; the other is that we need to match and cancel out all of the singular terms.

For simplicity, we just perform some typical and easy cases.\\

\noindent \textbf{Case 1}: $k_2=4$, $k_1=1$, $\gamma=2$ and $\mathcal{M}_\tau(W)=0$ for every $y\in \mathbb{T}^d$.

By viewing $y=x/\varepsilon$, $\tau=t/\varepsilon^{4}$, we look for $u_\varepsilon$ in the form
\begin{equation}\begin{aligned}
u_\varepsilon=&u_0+\varepsilon u_1+\varepsilon^2u_2+\cdots,\quad u_j=u_j(x,t,y,\tau),\\
&u_j\ \text{being 1-periodic in }(y,\tau).\end{aligned}\end{equation}
Moreover, we may expect that $u_0\equiv u_0(x,t)$ is independent of $y$ and $\tau$. Substituting the expression $(2.1)$ into the equation $(1.1)$ and matching the order of $\varepsilon$ yield that
\begin{equation}\begin{aligned}
&\varepsilon^{-4}:\partial_\tau u_0=0;\\
&\varepsilon^{-3}:\partial_\tau u_1=0;\\
&\varepsilon^{-2}:\partial_\tau u_2-\Delta_yu_0-Wu_0=0;\\
&\varepsilon^{-1}:\partial_\tau u_3-\Delta_yu_1-\operatorname{div}_x\nabla_yu_0-\operatorname{div}_y\nabla_xu_0-Wu_1=0;\\
&\varepsilon^{0}:\partial_tu_0-\Delta_xu_0+\partial_\tau u_4-\Delta_yu_2-\operatorname{div}_x\nabla_yu_1-\operatorname{div}_y\nabla_xu_1-Wu_2=f.
\end{aligned}\end{equation}
From $(2.2)_1$ and $(2.2)_3$, we know that $u_0\equiv u_0(x,t,y)$ satisfies
$$\Delta_yu_0+\mathcal{M}_\tau(W)u_0=0.$$
In general, if $\mathcal{M}_\tau(W)$ is not a constant for every $y\in \mathbb{T}^d$ (otherwise, which must be 0 since $\mathcal{M}(W)=0$), then the equation above is not solvable for $y\in\mathbb{T}^d$. Now if $\mathcal{M}_\tau(W)=0$ for every $y\in \mathbb{T}^d$, we know $u_0\equiv u_0(x,t)$. According to $(2.2)_3$ again, we have
$$u_2=u_0\cdot\int_0^{\tau}W(y,s)ds+\tilde{u}_2(x,t,y).$$

From $(2.2)_2$, $(2.2)_4$ and $\mathcal{M}_\tau(W)=0$ for every $y\in \mathbb{T}^d$, we know that $u_1\equiv u_1(x,t)$. Moreover, from $(2.2)_5$, we can determine that $u_0$ satisfies the following effective equation:
$$\partial_tu_0-\Delta_xu_0-\bar{W}u_0=f,$$
with $$\begin{aligned}\bar{W}&=:\int_{\mathbb{T}^{d+1}}W(y,\tau)\int_0^{\tau}W(y,s)dsdyd\tau\\
&=\frac12\int_{\mathbb{T}^{d+1}}\partial_\tau \left(\int_0^{\tau}W(y,s)ds\right)^2dyd\tau\\
&=\frac12\int_{\mathbb{T}^{d}} \left(\int_0^{1}W(y,\tau)d\tau\right)^2dy=0.\end{aligned}$$

According to $\bar{W}=0$, we know that $\gamma=2$ is not the correct value in this case since the limiting effect of the highly oscillating term is trivial. And we need to enlarge the value of $\gamma$ in this case.\\

\noindent \textbf{Case 2}: $k_1=k_2=2$, $\gamma=\frac{3}{2}$ and $\mathcal{M} (W)=0$.

By viewing $y=x/\varepsilon^{2}$, $\tau=t/\varepsilon^{2}$, we look for $u_\varepsilon$ in the form
\begin{equation}\begin{aligned}
u_\varepsilon=&u_0+\varepsilon u_1+\varepsilon^2u_2+\cdots+\varepsilon^{\frac12}v_1+\varepsilon^{\frac32}v_2+\cdots,\\
&u_j=u_j(x,t,y,\tau),\ v_j=v_j(x,t,y,\tau),\\
&\ \quad u_j,v_j\ \text{being 1-periodic in }(y,\tau).\end{aligned}\end{equation}

Substituting the expression $(2.3)$ into the equation $(1.1)$ and matching the order of $\varepsilon$ yield that
\begin{equation}\begin{aligned}
&\varepsilon^{-4}:-\Delta_y u_0=0;\\
&\varepsilon^{-3}:-\Delta_y u_1=0;\\
&\varepsilon^{-2}:\partial_\tau u_0-\Delta_y u_2-\operatorname{div}_x\nabla_yu_0-\operatorname{div}_y\nabla_xu_0=0;\\
&\varepsilon^{-1}:\partial_\tau u_1-\Delta_yu_3-\operatorname{div}_x\nabla_yu_1-\operatorname{div}_y\nabla_xu_1-Wv_1=0;\\
&\varepsilon^{0}:\partial_tu_0-\Delta_xu_0+\partial_\tau u_2-\Delta_yu_4-\operatorname{div}_x\nabla_yu_2-\operatorname{div}_y\nabla_xu_2-Wv_2=f,
\end{aligned}\end{equation}
and
\begin{equation}\begin{aligned}
&\varepsilon^{-\frac72}:-\Delta_y v_1=0;\\
&\varepsilon^{-\frac52}:-\Delta_y v_2=0;\\
&\varepsilon^{-\frac32}:\partial_\tau v_1-\Delta_y v_3-\operatorname{div}_x\nabla_yv_1-\operatorname{div}_y\nabla_xv_1-Wu_0=0;\\
&\varepsilon^{-\frac12}:\partial_\tau v_2-\Delta_yv_4-\operatorname{div}_x\nabla_yv_2-\operatorname{div}_y\nabla_xv_2-Wu_1=0;\\
\end{aligned}\end{equation}

\noindent
From $(2.4)_1$ and $(2.4)_3$, we know that $u_0\equiv u_0(x,t)$. Then, according to $(2.5)_1$ and $(2.5)_3$, we know that $v_1$ satisfies
$$\partial_\tau v_1-\mathcal{M}_y(W)u_0=0.$$
The equation above is solvable due to $u_0\equiv u_0(x,t)$ and $\mathcal{M}_\tau(\mathcal{M}_y(W))=0$. Then $v_1$ can be expressed as $$v_1=u_0\cdot\int_0^\tau\mathcal{M}_y(W)(s)ds+\tilde{v}_1(x,t).$$

\noindent Back to $(2.4)_4$ after in view of $(2.4)_2$, $u_1$ satisfies
$$\partial_\tau u_1-\mathcal{M}_y(W)v_1=0.$$

\noindent The equation above is solvable if and only if
$$0=\int_0^1\mathcal{M}_y(W)(\tau)\int_0^\tau\mathcal{M}_y(W)(s)dsd\tau=\frac12 \left(\int_0^1\mathcal{M}_y(W)(\tau)d\tau\right)^2.$$

\noindent Then $u_1$ can be expressed as
$$u_1=u_0\cdot\int_0^\tau \mathcal{M}_y(W)(s)\int_0^s\mathcal{M}_y(W)(t)dtds+\tilde{v}_1\int_0^\tau \mathcal{M}_y(W)(s)ds+\tilde{u}_1(x,t).$$

\noindent It follows from $(2.5)_2$ and $(2.5)_4$ that $v_2$ satisfies
$$\partial_\tau v_2-\mathcal{M}_y(Wu_1)=0.$$
The equation above is solvable if and only if
\begin{equation}\begin{aligned}0&=\int_0^1\mathcal{M}_y(W)(\tau)\int_0^\tau\mathcal{M}_y(W)(s)\int_0^s\mathcal{M}_y(W)(t)dtdsd\tau\\
&=\frac12 \int_0^1\mathcal{M}_y(W)(\tau)\left(\int_0^\tau\mathcal{M}_y(W)(s)ds\right)^2d\tau\\
&=\frac16\left(\int_0^1\mathcal{M}_y(W)(\tau)d\tau\right)^3.\end{aligned}\end{equation}

\noindent Then $v_2$ can be expressed as
$$v_2=\int_0^\tau\mathcal{M}_y(Wu_1)(s)ds+\tilde{v_2}(x,t).$$
Similar to the computation as in $(2.6)$, a direct computation yields that
$$\mathcal{M}(Wv_2)=0.$$
Consequently, $u_0$ satisfies the following effective equation:
$$\partial_tu_0-\Delta_xu_0=f.$$

Similar to the explanation in Case 1, we know that $\gamma=3/2$ is not the correct value in this case since the limiting effect of the highly oscillating term is trivial. And we need to enlarge the value of $\gamma$ in this case.\\

\noindent \textbf{Case 3}: $k_1=1$, $k_2=2$,  $\gamma=1$ and $\mathcal{M}(W)=0$.

By viewing $y=x/\varepsilon$, $\tau=t/\varepsilon^{2}$, we look for $u_\varepsilon$ in the form
\begin{equation}\begin{aligned}
u_\varepsilon=&u_0+\varepsilon u_1+\varepsilon^2u_2+\cdots,\quad u_j=u_j(x,t,y,\tau),\\
&u_j\ \text{being 1-periodic in }(y,\tau).\end{aligned}\end{equation}
Substituting the expression $(2.7)$ into the equation $(1.1)$ and matching the order of $\varepsilon$ yield that
\begin{equation}\begin{aligned}
&\varepsilon^{-2}:\partial_\tau u_0-\Delta_y u_0=0;\\
&\varepsilon^{-1}:\partial_\tau u_1-\Delta_y u_1-\operatorname{div}_y\nabla_xu_0-\operatorname{div}_x\nabla_yu_0-Wu_0=0;\\
&\varepsilon^{0}:\partial_t u_0-\Delta_x u_0+\partial_\tau u_2-\Delta_y u_2-\operatorname{div}_y\nabla_xu_1-\operatorname{div}_x\nabla_yu_1-Wu_1=f.\\
\end{aligned}\end{equation}
According to $(2.8)_1$, we know that $$u_0\equiv u_0(x,t).$$
Then $(2.8)_2$ reduces to
\begin{equation}
\partial_\tau u_1-\Delta_y u_1-Wu_0=0,
\end{equation}
which is solvable due to $\mathcal{M}(W)=0$ and $u_0\equiv u_0(x,t)$. We introduce the corrector $\chi_w(y,\tau)$ solving the following equation
\begin{equation}\left\{\begin{aligned}
\partial_\tau \chi_w-\Delta_y \chi_w=W\text{ in }\mathbb{T}^{d+1},\\
\chi_w\text{ is 1-periodic},\ \mathcal{M}(\chi_w)=0.
\end{aligned}\right.\end{equation}
Then $u_1$ can be expressed as
\begin{equation}
u_1=\chi_w\cdot u_0+\tilde{u}_1(x,t).
\end{equation}
Consequently, due to $(2.8)_3$, we conclude that $u_0$ satisfies the following homogenized equation:
\begin{equation}
\partial_t u_0-\Delta_x u_0-\mathcal{M}(W\chi_w)u_0=f.
\end{equation}
Note that $\mathcal{M}(W\chi_w)\neq 0$, then $\gamma=1$ is the correct value in the case.\\

\noindent \textbf{Case 4}: $k_1=1$, $k_2=3$, $\gamma=1$  and $\mathcal{M}(W)=0$.

By viewing $y=x/\varepsilon$, $\tau=t/\varepsilon^{3}$, we look for $u_\varepsilon$ in the form
\begin{equation}\begin{aligned}
u_\varepsilon=&u_0+\varepsilon u_1+\varepsilon^2u_2+\cdots,\quad u_j=u_j(x,t,y,\tau),\\
&u_j\ \text{being 1-periodic in }(y,\tau).\end{aligned}\end{equation}
Substituting the expression $(2.13)$ into the equation $(1.1)$ and matching the order of $\varepsilon$ yield that
\begin{equation}\begin{aligned}
&\varepsilon^{-3}:\partial_\tau u_0=0;\\
&\varepsilon^{-2}:\partial_\tau u_1-\Delta_y u_0=0;\\
&\varepsilon^{-1}:\partial_\tau u_2-\Delta_y u_1-\operatorname{div}_y\nabla_xu_0-\operatorname{div}_x\nabla_yu_0-Wu_0=0;\\
&\varepsilon^{0}:\partial_tu_0-\Delta_xu_0+\partial_\tau u_3-\Delta_y u_2-\operatorname{div}_y\nabla_xu_1-\operatorname{div}_x\nabla_yu_1-Wu_1=f.
\end{aligned}\end{equation}
First, according to $(2.14)_1$ and $(2.14)_2$, we know $$u_0\equiv u_0(x,t).$$
Next, due to $(2.14)_2$ and $(2.14)_3$, we know that $u_1$ satisfies
\begin{equation}
\Delta_yu_1+\mathcal{M}_\tau(W)u_0=0,
\end{equation}
which is solvable due to $\mathcal{M}_y(\mathcal{M}_\tau(W))=0$ and $u_0\equiv u_0(x,t)$. Moreover, we introduce the corrector $\chi_y(y)$ solving the following equation
\begin{equation}\left\{\begin{aligned}
\Delta_y \chi_y=-\mathcal{M}_\tau(W)\quad\text{ in }\quad\mathbb{T}^{d},\\
\chi_y\text{ is 1-periodic},\ \mathcal{M}_y(\chi_y)=0.
\end{aligned}\right.\end{equation}
Then $u_1$ can be expressed as
\begin{equation}
u_1=\chi_y\cdot u_0+\tilde{u}_1(x,t).
\end{equation}
Consequently, due to $(2.14)_4$, we conclude that $u_0$ satisfies the following homogenized equation:
\begin{equation}
\partial_t u_0-\Delta_x u_0-\mathcal{M}(W\chi_y)u_0=f.
\end{equation}
Note that $\mathcal{M}(W\chi_y)\neq 0$, then $\gamma=1$ is the correct value in the case.\\

\noindent \textbf{Case 5}: $k_1=1$, $k_2=1$, $\gamma=\frac12$ and $\mathcal{M}(W)=0$.\\

By viewing $y=x/\varepsilon$, $\tau=t/\varepsilon$, we look for $u_\varepsilon$ in the form
\begin{equation}\begin{aligned}
u_\varepsilon=&u_0+\varepsilon u_1+\varepsilon^2u_2+\cdots+\varepsilon^{\frac12}v_1+\varepsilon^{\frac32}v_2+\cdots,\\
&u_j=u_j(x,t,y,\tau),\ v_j=v_j(x,t,y,\tau),\\
&\ \quad u_j,v_j\ \text{being 1-periodic in }(y,\tau).\end{aligned}\end{equation}
Substituting the expression $(2.19)$ into the equation $(1.1)$ and matching the order of $\varepsilon$ yield that
\begin{equation}\begin{aligned}
&\varepsilon^{-2}:-\Delta_y u_0=0;\\
&\varepsilon^{-1}:\partial_\tau u_0-\Delta_y u_1-\operatorname{div}_y\nabla_xu_0-\operatorname{div}_x\nabla_yu_0=0;\\
&\varepsilon^{0}:\partial_tu_0-\Delta_xu_0+\partial_\tau u_1-\Delta_y u_2-\operatorname{div}_y\nabla_xu_1-\operatorname{div}_x\nabla_yu_1-Wv_1=f,\\
\end{aligned}\end{equation}
and
\begin{equation}\begin{aligned}
&\varepsilon^{-\frac32}:-\Delta_y v_1=0;\\
&\varepsilon^{-\frac12}:\partial_\tau v_1-\Delta_y v_2-\operatorname{div}_y\nabla_xv_1-\operatorname{div}_x\nabla_yv_1-Wu_0=0.\\
\end{aligned}\end{equation}
First, according to $(2.20)_1$ and $(2.20)_2$, we know $$u_0\equiv u_0(x,t).$$
Next, due to $(2.21)_1$ and $(2.21)_2$, we know that $v_1$ satisfies
\begin{equation}
\partial_\tau v_1-\mathcal{M}_y(W)u_0=0,
\end{equation}
which is solvable due to $\mathcal{M}_\tau(\mathcal{M}_y(W))=0$ and $u_0\equiv u_0(x,t)$.
Then $v_1$ can be expressed as
\begin{equation}
v_1=u_0\cdot \int_0^\tau \mathcal{M}_y(W)(s)ds+\tilde{v}_1(x,t).
\end{equation}
Consequently, due to $(2.20)_3$, we conclude that $u_0$ satisfies the following homogenized equation:
\begin{equation}
\partial_t u_0-\Delta_x u_0-\bar{W}u_0=f,
\end{equation}\\
where $$\begin{aligned}\bar{W}&=:\int_{\mathbb{T}^{d+1}}W(y,\tau)\int_0^{\tau}\mathcal{M}_y(W)(s)dsdyd\tau\\
&=\frac12\int_{\mathbb{T}}\partial_\tau \left(\int_0^{\tau}\mathcal{M}_y(W)(s)ds\right)^2dyd\tau\\
&=\frac12 \left(\int_0^{1}\mathcal{M}_y(W)(\tau)d\tau\right)^2=0.\end{aligned}$$
Similar to the explanation in case 1, we know that $\gamma=1/2$ is not the correct value in this case since the limiting effect of the highly oscillating term is trivial. And we need to enlarge the value of $\gamma$ in this case.\\

\noindent \textbf{Case 6}: $k_1=1$, $k_2=1$, $\gamma=1$ and $\mathcal{M}_y(W)=0$ for every $\tau\in \mathbb{T}$.

By viewing $y=x/\varepsilon$, $\tau=t/\varepsilon$, we look for $u_\varepsilon$ in the form
\begin{equation}\begin{aligned}
u_\varepsilon=&u_0+\varepsilon u_1+\varepsilon^2u_2+\cdots,\quad u_j=u_j(x,t,y,\tau),\\
&u_j\ \text{being 1-periodic in }(y,\tau).\end{aligned}\end{equation}
Substituting the expression $(2.25)$ into the equation $(1.1)$ and matching the order of $\varepsilon$ yield that
\begin{equation}\begin{aligned}
&\varepsilon^{-2}:-\Delta_y u_0=0;\\
&\varepsilon^{-1}:\partial_\tau u_0-\Delta_y u_1-\operatorname{div}_y\nabla_xu_0-\operatorname{div}_x\nabla_yu_0-Wu_0=0;\\
&\varepsilon^{0}:\partial_tu_0-\Delta_xu_0+\partial_\tau u_1-\Delta_y u_2-\operatorname{div}_y\nabla_xu_1-\operatorname{div}_x\nabla_yu_1-Wu_1=f.\\
\end{aligned}\end{equation}

First, according to $(2.26)_1$ and $(2.26)_2$, we know $u_0$ satisfies
$$\partial_\tau u_0-\mathcal{M}_y(W)u_0=0.$$
In general, the equation above is not solvable unless $\mathcal{M}_y(W)=0$ for every $\tau\in \mathbb{T}$. Moreover, even if there exist $W$ and $u_0$ satisfying the equation above, we know that $u_0=u_0(x,t,\tau)$, which may be not expected since we expect that $u_0=u_0(x,t)$ satisfies the homogenized equation. Then, we assume that $\mathcal{M}_y(W)=0$ for every $\tau\in \mathbb{T}$, which immediately implies that $u_0\equiv u_0(x,t)$.

Next, due to $(2.26)_2$, we know that $u_1$ satisfies
\begin{equation}
\Delta_yu_1+Wu_0=0,
\end{equation}
which is solvable due to $\mathcal{M}_y(W)=0$ for every $\tau\in \mathbb{T}$ and $u_0\equiv u_0(x,t)$. For every $\tau\in\mathbb{T}$, we introduce the corrector $\chi_\tau(y,\tau)$ solving the following equation
\begin{equation}\left\{\begin{aligned}
\Delta_y \chi_\tau=-W\quad\quad\text{in }\mathbb{T}^{d},\\
\chi_\tau\text{ is 1-periodic},\ \mathcal{M}_y(\chi_\tau)(\tau)=0.
\end{aligned}\right.\end{equation}
Then $u_1$ can be expressed as
\begin{equation}
u_1=\chi_\tau\cdot u_0+\tilde{u}_1(x,t,\tau).
\end{equation}
Consequently, due to $(2.26)_3$, we conclude that $u_0$ satisfies the following homogenized equation:
\begin{equation}
\partial_t u_0-\Delta_x u_0-\mathcal{M}(W\chi_\tau)u_0=f.
\end{equation}
Note that $\mathcal{M}(W\chi_\tau)\neq 0$, then $\gamma=1$ is the correct value in this case.\\

In conclusion, we investigate the parabolic equation with time varying potentials in periodic homogenization under 5 different conditions (for simplicity, we assume $k_1=1$ and $k_2=k$), which are listed as below: \\

\noindent \textbf{Assumption 1}: $3\geq k>2$, $\gamma=k-1$ and $\mathcal{M}_\tau(W)(y)=0$ for every $y\in \mathbb{T}^d$.\\

\noindent \textbf{Assumption 2}: $1< k<2$, $\gamma=1$ and $\mathcal{M}(W)=0$.\\

\noindent \textbf{Assumption 3}: $k=2$, $\gamma=1$ and $\mathcal{M}(W)=0$.\\

\noindent \textbf{Assumption 4}: $k>2$, $\gamma=1$ and $\mathcal{M}(W)=0$.\\

\noindent \textbf{Assumption 5}: $0\leq k\leq 1$, $\gamma=1$ and $\mathcal{M}_y(W)=0$ for every $y\in\mathbb{T}^d$.\\

It seems that we forget the case for $k>3$. In fact, in view of the computation in Case 1, we additionally need to assume that $\mathcal{M}_\tau(W)(y)=0$ for every $y\in\mathbb{T}^d$ and $\gamma>1$, as in Assumption 1. However, the effective equation turns out to be the heat equation without a nontrivial potential for any $\gamma\in (0,\infty)$, which implies that there is no correct value for $\gamma$ in this case. See Remark 3.7 for more details.

\section{Main Results}
In this section, we introduce our main results, including the homogenization results and the convergence rates, under different assumptions at the end of Section 2. Note that in the following theorems, we always assume that $\Omega$ is a bounded $C^{1,1}$ domain in $\mathbb{R}^d$ with $d\geq 2$, $T\in (0,\infty)$
is a positive constant, and $k_1=1$, $k_2=k\in [0,\infty)$.

\begin{thm}
Assume that $k=2$, $\gamma=1$, $W\in L^\infty(\mathbb{T}^{d+1})$ with $\mathcal{M}(W)=0$, $f\in L^2(\Omega_T)$ and $g\in H^1_0(\Omega)$, then the homogenized equation for $(1.1)$ is given by
\begin{equation}\left\{\begin{aligned}
\partial_t {u}_0-\Delta {u}_0 -\mathcal{M}(\chi_1W){u}_0&={f}, \quad\quad\text{in }\Omega\times (0,T),\\
{u}_0&=0, \quad\quad\text{on }\partial\Omega\times (0,T),\\
{u}_0&=g, \quad\quad\text{on }\Omega\times \{t=0\},
\end{aligned}\right.\end{equation}
where the corrector $\chi_1(y,\tau)$ is defined as
\begin{equation}\left\{\begin{aligned}
\partial_\tau\chi_1 -\Delta_y\chi_1=W\quad\text{ in }\quad\mathbb{T}^{d+1},\\
\chi_1\text{ is 1-periodic with }\mathcal{M}(\chi_1)=0.
\end{aligned}\right.\end{equation}
Moreover, we have the following convergence rates estimates:
\begin{equation}||u_\varepsilon-u_0||_{L^\infty_TL^2_x}\leq C\varepsilon \left(||f||_{L^2(\Omega_T)}+||g||_{H^1(\Omega)}\right),\end{equation}
where the constant $C$ depends only on $||W||_{L^\infty(\mathbb{T}^{d+1})}$, $d$, $\Omega$ and $T$.
\end{thm}

\begin{thm}
Assume that $k>2$, $\gamma=1$, $W\in L^\infty(\mathbb{T}^{d+1})$ with $\mathcal{M}(W)=0$ as well  as $\nabla_y\left(W-\mathcal{M}_\tau(W)\right)\in L^\infty(\mathbb{T}^{d+1})$, $f\in L^2(\Omega_T)$ and $g\in H^1_0(\Omega)$, then the homogenized equation for $(1.1)$ is given by
\begin{equation}\left\{\begin{aligned}
\partial_t {u}_0-\Delta {u}_0+\mathcal{M}(\chi_{2}W){u}_0&={f}, \quad\quad\text{in }\Omega\times (0,T),\\
{u}_0&=0, \quad\quad\text{on }\partial\Omega\times (0,T),\\
{u}_0&=g, \quad\quad\text{on }\Omega\times \{t=0\}.
\end{aligned}\right.\end{equation}
where the corrector $\chi_2(y)$ is defined as
\begin{equation}\left\{\begin{aligned}
\Delta_y\chi_{2}=\mathcal{M}_\tau(W)(y)\quad\quad\text{ in }\quad\mathbb{T}^{d},\\
\chi_{2}\text{ is 1-periodic with }\mathcal{M}_y(\chi_{2})=0,
\end{aligned}\right.\end{equation}
Moreover, we have the following convergence rates estimates:
\begin{equation}||u_\varepsilon-u_0||_{L^\infty_TL^2_x}\leq C\max(\varepsilon^{k-2},\varepsilon) \left(||f||_{L^2(\Omega_T)}+||g||_{H^1(\Omega)}\right),\end{equation}
where the constant $C$ depends only on $W$, $d$, $\Omega$ and $T$.
\end{thm}

\begin{thm}
Assume that $1<k<2$, $\gamma=1$, $W\in L^\infty(\mathbb{T}^{d+1})$ and $\partial_\tau\left(W-\mathcal{M}_y(W)\right)\in L^\infty(\mathbb{T}^{d+1})$ with $\mathcal{M}(W)=0$, $f\in L^2(\Omega_T)$ and $g\in H^1_0(\Omega)$, then the homogenized equation for $(1.1)$ is given by
\begin{equation}\left\{\begin{aligned}
\partial_t {u}_0-\Delta {u}_0+\mathcal{M}(\chi_3 W)u_0&=f, \quad\quad\text{in }\Omega\times (0,T),\\
{u}_0&=0, \quad\quad\text{on }\partial\Omega\times (0,T),\\
{u}_0&=g, \quad\quad\text{on }\Omega\times \{t=0\}.
\end{aligned}\right.\end{equation}

where, for every $\tau\in \mathbb{T}$, the corrector $\chi_3(y,\tau)$ is defined as,
\begin{equation}\left\{\begin{aligned}
\Delta_y \chi_3=W-\mathcal{M}_y(W)\quad\text{ in }\quad\mathbb{T}^{d},\\
\chi_3\text{ is 1-periodic},\ \mathcal{M}_y(\chi_3)(\tau)=0.
\end{aligned}\right.\end{equation}
And we have the following convergence rates estimates:
\begin{equation}||u_\varepsilon-u_0||_{L^\infty_TL^2_x}\leq C\max(\varepsilon^{2-k},\varepsilon^{k-1}) \left(||f||_{L^2(\Omega_T)}+||g||_{H^1(\Omega)}\right),\end{equation}
where the constant $C$ depends only on $W$, $T$, $d$, $\Omega$ and $k$. Moreover,
the constant $C$ will blow up as $k\rightarrow 1$.
\end{thm}

\begin{thm}
Assume that $0\leq k\leq 1$, $\gamma=1$, $W,\partial_\tau W\in L^\infty(\mathbb{T}^{d+1})$ with $\mathcal{M}_y(W)(\tau)=0$ for every $\tau\in \mathbb{T}$, $f\in L^2(\Omega_T)$ and $g\in H^1_0(\Omega)$, then for $0<k\leq 1$, the homogenized equation for $(1.1)$ is given by
\begin{equation}\left\{\begin{aligned}
\partial_t {u}_0-\Delta {u}_0-\mathcal{M}(\chi_{3} W)u_0&=f, \quad\quad\text{in }\Omega\times (0,T),\\
{u}_0&=0, \quad\quad\text{on }\partial\Omega\times (0,T),\\
{u}_0&=g, \quad\quad\text{on }\Omega\times \{t=0\}.
\end{aligned}\right.\end{equation}
with the same corrector $\chi_3(y,\tau)$ defined in $(3.8)$. Moreover, we have the following convergence rates estimates:
\begin{equation}||u_\varepsilon-u_0||_{L^\infty_TL^2_x}\leq C\varepsilon^k \left(||f||_{L^2(\Omega_T)}+||g||_{H^1(\Omega)}\right),\end{equation}
where the constant $C$ depends only on $W$, $d$, $\Omega$ and $T$. For $k=0$, the homogenized equation for $(1.1)$ is given by
\begin{equation}\left\{\begin{aligned}
\partial_t {u}_0-\Delta {u}_0-\mathcal{M}_y(\chi_{3} W)(t)\cdot u_0&=f, \quad\quad\text{in }\Omega\times (0,T),\\
{u}_0&=0, \quad\quad\text{on }\partial\Omega\times (0,T),\\
{u}_0&=g, \quad\quad\text{on }\Omega\times \{t=0\}
\end{aligned}\right.\end{equation}
and  we have the following convergence rates estimates:
\begin{equation}||u_\varepsilon-u_0||_{L^\infty_TL^2_x}\leq C\varepsilon \left(||f||_{L^2(\Omega_T)}+||g||_{H^1(\Omega)}\right),\end{equation}
where the constant $C$ depends only on $W$, $d$, $\Omega$ and $T$.
\end{thm}
\begin{rmk}
In Theorem 3.4, compared with the conditions satisfied by the potential $W$ in Theorem 3.3, we need a more stronger conditions: $W\in L^\infty(\mathbb{T}^{d+1})$ with $\mathcal{M}_y(W)(\tau)=0$ for every $\tau\in \mathbb{T}$, since the singular term $\varepsilon^{-1}\mathcal{M}_y(W)(\tau) u_\varepsilon$ is too bad for $0\leq k\leq 1$ and we can not even have a uniform estimate for $u_\varepsilon$, which is solution to the following parabolic equation
\begin{equation*}\left\{\begin{aligned}
\partial_t u_\varepsilon-\Delta u_\varepsilon -\varepsilon ^{-1} \mathcal{M}_y(W)\left( t /{\varepsilon^{k}}\right)u_\varepsilon&=f, \quad\quad\text{in }\Omega\times (0,T),\\
u_\varepsilon&=0, \quad\quad\text{on }\partial\Omega\times (0,T),\\
u_\varepsilon&=g, \quad\quad\text{on }\Omega\times \{t=0\},
\end{aligned}\right.\end{equation*}
if $\mathcal{M}(W)=0$ with $\mathcal{M}_y(W)(\tau)\not\equiv 0$, which is congruent to that the constant $C$ in Theorem 3.3 will blow up as $k\rightarrow 1$.
\end{rmk}

\begin{thm}
Assume that $3\geq k>2$, $\gamma=k-1$, $W\in L^\infty(\mathbb{T}^{d+1})$, $\nabla_yW,\nabla ^2_yW,\nabla ^3_yW\in L^\infty(\mathbb{T}^{d+1})$ with $\mathcal{M}_\tau(W)(y)=0$ for every $y\in \mathbb{T}^d$, $f\in L^2(\Omega_T)$ and $g\in H^1_0(\Omega)$, then the homogenized equation for $(1.1)$ is given by
\begin{equation}\left\{\begin{aligned}
\partial_t u_0-\Delta u_0+\mathcal{M}(\nabla_y {\chi_4}\nabla_yW) u_0&=f, \quad\quad\text{in }\Omega\times (0,T),\\
u_0&=0, \quad\quad\text{on }\partial\Omega\times (0,T),\\
u_0&=g, \quad\quad\text{on }\Omega\times \{t=0\}.
\end{aligned}\right.\end{equation}
where the corrector $\chi_4(y,\tau)$ is defined as
\begin{equation}
\chi_4(y,\tau)=:\int_0^\tau \tilde{\chi_5}(y,s)ds,
\end{equation}
where
\begin{equation}\tilde{\chi_5}(y,\tau)=:{\chi_5}(y,\tau)-\mathcal{M}_\tau(\chi_5)(y),\end{equation}
with
\begin{equation}{\chi_5}(y,\tau)=:\int_0^\tau W(y,s)ds.\end{equation}
Moreover, we have the following convergence rates estimates:
\begin{equation}||u_\varepsilon-u_0||_{L^\infty_TL^2_x}\leq C \varepsilon^{k-2}\left(||f||_{L^2(\Omega_T)}+||g||_{H^1(\Omega)}\right),\end{equation}
where the constant $C$ depends only on $W$, $d$, $\Omega$ and $T$.
\end{thm}
\begin{rmk}
1: In general, $\mathcal{M}(\nabla_y {\chi_4}\nabla_yW)\neq 0$. For example, we choose $W(y,\tau)=\sin(2\pi \tau)V(y)$ for any non-constant periodic function $V\in C^3(\mathbb{T}^d)$. Then, the $W$ we define satisfies the assumptions in Theorem 3.6. And it is easy to see that
 $$\tilde{\chi_5}(y,\tau)=-\frac 1{2\pi}\cos(2\pi \tau)V(y)\text{ and }\chi_4(y,\tau)=-\frac 1{4\pi^2}sin(2\pi \tau)V(y).$$
 Therefore, $$\mathcal{M}(\nabla_y {\chi_4}\nabla_yW)=-\frac1{4\pi^2}\int_\mathbb{T}\sin^2(2\pi \tau)d\tau\cdot\int_{\mathbb{T}^d}|\nabla_y V|^2dy\neq 0.$$

2: Similar to the explanation in Remark 3.5, the singular term $\varepsilon^{1-k}(W-\mathcal{M}_y(W)(\tau))u_\varepsilon$ is too bad for $2<k\leq 3$ if we only we $\mathcal{M}(W)=0$ with $\mathcal{M}_\tau(W)(y)\not\equiv 0$.

3: If $\infty>k>3$, $\mathcal{M}_\tau(W)(y)=0$ for every $y\in \mathbb{T}^d$ with suitable smoothing conditions satisfied by the potential $W$, then for any positive constant  $\gamma\in (0,\infty)$, the effective equation for $(1.1)$ is the heat equation with a trivial potential, which implies that there is no correct value for $\gamma$, such that the limiting effect of the highly oscillating term is nontrivial. To see this homogenization result, one may consult the proof of Theorem 3.3 (especially, $(5.6)$-$(5.13)$ in Section 5) for an elementary comprehension.
\end{rmk}

\begin{rmk}
The convergence rates obtained in Theorems 3.1-3.3 are congruent to the rates given by the two-scale expansions. Actually, in the process of two-scale expansions, we need to match and cancel out the terms of every negative order of $\varepsilon$, therefore, the rates given by the two-scale expansions should be the smallest positive order of $\varepsilon$, see \cite{MR2839402} more details. Moreover, under some better conditions satisfied by the potential $W$, we can eliminate the influence from the other part. For example, at a first glance, the rates given by the two-scale expansions in Theorem 3.6 should be $O\left(\max(\varepsilon^{k-2},\varepsilon^{3-k})\right)$. However, in view of $(6.44)$ and the content after $(6.44)$, the condition $\mathcal{M}_\tau(W)(y)=0$ for every $y\in \mathbb{T}^d$ would eliminate the influence from the terms of size $O(\varepsilon^{3-k})$ for $2<k<3$. Consequently, we come to the final convergence rates $(3.18)$.
\end{rmk}

\begin{rmk}
In view of the convergence rates obtained in \cite{MR4072212} for the following parabolic operator with  non-self-similar scales:
$$\partial_t-\operatorname{div}\left(A(x/\varepsilon,t/\varepsilon^k)\right).$$
Then it is an interesting problem to consider this parabolic operator considered in \cite{MR4072212} with time-varying potentials (even with different scales compared to the leading coefficient $A$), which may be left for further.
\end{rmk}
\section{Proofs of Theorem 3.1 and Theorem 3.2}
Before we prove Theorem 3.1 and Theorem 3.2, we first introduce some useful results.
\begin{lemma}Let $0<T<\infty$ be a positive constant and  $\Omega$ be a bounded $C^{1,1}$ domain in $\mathbb{R}^d$ with $d\geq 2$. Let $u(x,t)\in L^2(0,T;H_0^1(\Omega))$ be the weak solution to the following parabolic equation:
\begin{equation*}\left\{\begin{aligned}
\partial_t u-\Delta u&=f, \quad\quad\text{in }\Omega\times (0,T),\\
u&=0, \quad\quad\text{on }\partial\Omega\times (0,T),\\
u&=g, \quad\quad\text{on }\Omega\times \{t=0\},
\end{aligned}\right.\end{equation*}
with $f\in L^2(0,T;H^{-1}(\Omega))$ and $g\in L^2(\Omega)$, then
$$||u||_{L^\infty_TL^2_x}+||\nabla u||_{L^2(\Omega_T)}\leq C\left(||g||_{L^2(\Omega)}+||f||_{L^2(0,T;H^{-1}(\Omega))}\right),$$
with $C$ being a universal positive constant. Moreover, if $f\in L^2(\Omega_T)$ and $g\in H^1_0(\Omega)$, then
$$||\nabla u||_{L^\infty_TL^2_x}+||\nabla^2 u||_{L^2(\Omega_T)}+||\partial_t u||_{L^2(\Omega_T)}\leq C\left(||g||_{H^1(\Omega)}+||f||_{L^2(\Omega_T)}\right),$$
with the positive constant $C$ depending only on $d$, $T$ and $\partial \Omega$.
\end{lemma}
\begin{proof}These estimates are classical in parabolic theory, we omit it for simplicity and refer to \cite{MR0181836,MR0241822} for the details.\end{proof}

\begin{lemma}
Let $k\in (0,\infty)$ be a positive constant and  $h_m=h_m(y,\tau)\in L^p(\mathbb{T}^{d+1})$ with $1<p<\infty$. Suppose that $||h_m||_{L^p(\mathbb{T}^{d+1})}\leq C$, and $\int_{\mathbb{T}^{d+1}}h_m\rightarrow \widehat{M_0}$ as $m\rightarrow\infty$. Then, for any $\varepsilon_m\rightarrow 0$, we have $h_m(x/\varepsilon_m,t/\varepsilon_m^k)\rightharpoonup \widehat{M_0}$ weakly in $L^p(Q)$ for and bounded domain $Q$ in $\mathbb{R}^{d+1}$. Moreover, if $p=\infty$, then $h_m(x/\varepsilon_m,t/\varepsilon_m^k)\rightharpoonup \widehat{M_0}$ weakly in $L^q(Q)$ for any $1<q<\infty$ and $h_m(x/\varepsilon_m,t/\varepsilon_m^k)\rightharpoonup \widehat{M_0}$ in $L^\infty(Q)$ weak star.
\end{lemma}
\begin{proof}
The proof is standard. One can find the proof in \cite[Lemma 3.4]{MR3361284} for $p=2$, and we omit the details for simplicity.
\end{proof}

The next lemma is due to J.P. Aubin and J.L. Lions, which is used to guarantee the strong convergence in $L^2(\Omega_T)$.

\begin{lemma}
Let $X_0\subset X\subset X_1$ be three Banach spaces. Suppose that $X_0,X_1$ are reflexive and that the injection $X_0\subset X$ is compact. Let $1<p_0,p_1<\infty$. Define
$$\mathcal{B}=\left\{u:u\in L^{p_0}(T_0,T_1;X_0)\text{ and }\partial_tu\in L^{p_1}(T_0,T_1;X_1)\right\}$$
with norm $$||u||_{\mathcal{B}}=||u||_{L^{p_0}(T_0,T_1;X_0)}+||\partial_tu||_{L^{p_1}(T_0,T_1;X_1)}.$$
Then, $\mathcal{B}$ is a Banach space, and the injection $\mathcal{B}\subset L^{p_0}(T_0,T_1;X)$ is compact.
\end{lemma}

\subsection{Proof of Theorem 3.1}
In this subsection, we always assume that $k=2$, $\gamma=1$, $W\in L^\infty(\mathbb{T}^{d+1})$ with $\mathcal{M}(W)=0$.

In order to prove the homogenization part, we need first to obtain the uniform estimates for $u_\varepsilon$, which is stated in the following lemma.
\begin{lemma}[Uniform estimates]
Under the conditions in Theorem 3.1, there hold the following uniform estimates for $u_\varepsilon$:
\begin{equation*}
||u_\varepsilon||_{L^\infty_TL^2_x}+||\nabla u_\varepsilon||_{L^2(\Omega_T)}+|| u_\varepsilon||_{L^2(\Omega_T)}\leq C\left(||g||_{L^2(\Omega)}+||f||_{L^2(\Omega)}\right),
\end{equation*}
for a positive constant $C$ depending only on $||W||_{L^\infty(\mathbb{T}^{d+1})}$, $d$ and $T$.

\end{lemma}
\begin{proof}Recall that $v_\varepsilon=u_\varepsilon e^{-Nt}$ and $\chi_1(y,\tau)$ satisfy the following parabolic equations:
\begin{equation}\left\{\begin{aligned}
\partial_t v_\varepsilon-\Delta v_\varepsilon+\left(N-\varepsilon^{-1}W^\varepsilon\right)v_\varepsilon&=\tilde{f}, \quad\quad\text{in }\Omega\times (0,T),\\
v_\varepsilon&=0, \quad\quad\text{on }\partial\Omega\times (0,T),\\
v_\varepsilon&=g, \quad\quad\text{on }\Omega\times \{t=0\},
\end{aligned}\right.\end{equation}
with $\tilde{f}=fe^{-Nt}$, and
\begin{equation}\left\{\begin{aligned}
\partial_\tau\chi_1 -\Delta_y\chi_1=W\quad\text{ in }\quad\mathbb{T}^{d+1},\\
\chi_1\text{ is 1-periodic with }\mathcal{M}(\chi_1)=0,
\end{aligned}\right.\end{equation}
respectively. According to the parabolic theory, we know that
$$\chi_1\in W^{2,1}_p(\mathbb{T}^{d+1}),\text{ for any }1<p<\infty,$$
and by embedding,
$$\chi_1 \in C^{1+\alpha,\frac{1+\alpha}2}(\mathbb{T}^{d+1}), \text{ for any }0<\alpha<1.$$
 Then, for $T\geq t>0$, a direct computation shows that
\begin{equation}
\int_{\Omega_t}\varepsilon^{-1}W^\varepsilon v_\varepsilon^2=\int_{\Omega_t}\left(\varepsilon v_\varepsilon^2\partial_t\chi_1^\varepsilon -\varepsilon v_\varepsilon^2\Delta_x\chi_1^\varepsilon\right)
=E_1+E_2.\end{equation}
It is easy to see that
\begin{equation}\begin{aligned}
E_1=&\int_{\Omega\times\{t\}}\varepsilon v_\varepsilon^2\chi_1^\varepsilon-\int_{\Omega\times\{0\}}\varepsilon g^2\chi_1^\varepsilon-2\int_{\Omega_t}\varepsilon v_\varepsilon\chi_1^\varepsilon\partial_t v_\varepsilon\\
=&E_{11}+E_{12}+E_{13}.
\end{aligned}\end{equation}

\noindent
A direct computation yields that
\begin{equation}\begin{aligned}
E_{13}=&-2\int_{\Omega_t}\varepsilon v_\varepsilon\chi_1^\varepsilon\left(\Delta v_\varepsilon+\varepsilon^{-1}W^\varepsilon v_\varepsilon-Nv_\varepsilon+\tilde{f} \right)\\
=&2\int_{\Omega_t}\varepsilon \chi_1^\varepsilon|\nabla v_\varepsilon|^2+2\int_{\Omega_t} v_\varepsilon\nabla_y\chi_1^\varepsilon\nabla v_\varepsilon-2\int_{\Omega_t}\chi_1^\varepsilon W^\varepsilon v_\varepsilon^2+2\int_{\Omega_t}\varepsilon v_\varepsilon\chi_1^\varepsilon\left(Nv_\varepsilon-\tilde{f} \right)
\end{aligned}\end{equation}

and
\begin{equation}
E_2=\int_{\Omega_t}\varepsilon \chi_1^\varepsilon|\nabla v_\varepsilon|^2+\int_{\Omega_t} v_\varepsilon\nabla_y\chi_1^\varepsilon\nabla v_\varepsilon.
\end{equation}
Consequently, basic energy estimates after combining $(4.3)$-$(4.6)$ yields that
\begin{equation*}
\int_{\Omega\times \{t\}} |v_\varepsilon|^2+\int_0^t\int_\Omega |\nabla v_\varepsilon|^2+\int_0^t\int_\Omega |v_\varepsilon|^2\leq C\int_0^t\int_\Omega |f|^2+C\int_\Omega |g|^2,\end{equation*}
for a positive constant $C$ depending only on $||W||_{L^\infty(\mathbb{T}^{d+1})}$ and $T$, which yields the desired estimates after noting $v_\varepsilon=u_\varepsilon e^{-Nt}$.

\end{proof}
Now we are ready to give the proof Theorem 3.1.\\
\noindent\textbf{Prove of Theorem 3.1.}\\

\noindent
First, using the same notations in Lemma 4.4, a direct computation shows that
\begin{equation}\begin{aligned}
&\varepsilon^{-1}W^\varepsilon v_\varepsilon =\varepsilon v_\varepsilon\partial_t\chi_1^\varepsilon -\varepsilon v_\varepsilon\Delta_x\chi_1^\varepsilon\\
 =&\varepsilon \partial_t (v_\varepsilon\chi_1^\varepsilon)-\varepsilon\Delta_x( v_\varepsilon\chi_1^\varepsilon)+2\nabla v_\varepsilon\nabla_y\chi_1^\varepsilon -\varepsilon\chi_1^\varepsilon (\partial_t v_\varepsilon-\Delta v_\varepsilon)\\
=&\varepsilon \partial_t (v_\varepsilon\chi_1^\varepsilon)-\varepsilon\Delta_x( v_\varepsilon\chi_1^\varepsilon)+2\nabla v_\varepsilon\nabla_y\chi_1^\varepsilon -\varepsilon\chi_1^\varepsilon(\varepsilon^{-1}W^\varepsilon v_\varepsilon-Nv_\varepsilon+\tilde{f})\\
=&\varepsilon \partial_t (v_\varepsilon\chi_1^\varepsilon)-\varepsilon\Delta_x( v_\varepsilon\chi_1^\varepsilon)+2\nabla v_\varepsilon\nabla_y\chi_1^\varepsilon -\chi_1^\varepsilon W^\varepsilon v_\varepsilon+\varepsilon\chi_1^\varepsilon N v_\varepsilon-\varepsilon\chi_1^\varepsilon \tilde{f}.
\end{aligned}\end{equation}
Therefore, by setting
$\tilde{v}_\varepsilon= v_\varepsilon -\varepsilon\chi_1^\varepsilon v_\varepsilon$, it is easy to check that $\tilde{v}_\varepsilon$ satisfies the following parabolic equation:
\begin{equation}\left\{\begin{aligned}
\partial_t \tilde{v}_\varepsilon-\Delta \tilde{v}_\varepsilon +N\tilde{v}_\varepsilon&=f_1(v_\varepsilon),& &\text{in }\Omega\times (0,T),\\
\tilde{v}_\varepsilon&=0,& &\text{on }\partial\Omega\times (0,T),\\
\tilde{v}_\varepsilon&=g-\varepsilon\chi_1^\varepsilon g,&\quad &\text{on }\Omega\times \{t=0\}.
\end{aligned}\right.\end{equation}
with $f_1(v_\varepsilon)=\tilde{f}+2\nabla v_\varepsilon\nabla_y\chi_1^\varepsilon -\chi_1^\varepsilon W^\varepsilon v_\varepsilon-\varepsilon\chi_1^\varepsilon \tilde{f}$.
Applying the $W^{2,1}_2$ theory, we have
\begin{equation}||\tilde{v}_\varepsilon||_{W^{2,1}_2(\Omega_T)}\leq C\left(||f||_{L^2(\Omega_T)}+||g||_{H^1(\Omega)}\right).\end{equation}
Moreover, it is easy to see that
$$||\partial_t \partial_i \tilde{v}_\varepsilon||_{L^2(0,T;H^{-1}(\Omega))}\leq C\left(||f||_{L^2(\Omega_T)}+||g||_{H^1(\Omega)}\right),$$
which combines with $(4.9)$ and Lemma 4.3 implies that there exists $v_0\in W^{2,1}_2(\Omega_T)$, such that
\begin{equation}\begin{aligned}
\tilde{v}_\varepsilon \rightharpoonup v_0\text{ weakly in }W^{2,1}_2(\Omega_T);\\
\tilde{v}_\varepsilon \rightarrow v_0\text{ strongly in }L^2(0,T;H^1(\Omega));\\
v_\varepsilon \rightarrow v_0\text{ strongly in }L^2(\Omega_T).
\end{aligned}\end{equation}
To proceed, note that
$$\begin{aligned}2\nabla v_\varepsilon\nabla_y\chi_1^\varepsilon&= 2\nabla \tilde{v}_\varepsilon \nabla_y\chi_1^\varepsilon+2\varepsilon \chi_1^\varepsilon \nabla v_\varepsilon \nabla_y\chi_1^\varepsilon+ 2v_\varepsilon |\nabla_y\chi_1^\varepsilon|^2\\
&\rightharpoonup 2v_0\mathcal{M}(|\nabla_y\chi_1|^2)=2v_0\mathcal{M}(\chi_1 W)\text{ weakly in }L^2(\Omega_T),
\end{aligned}$$
where the last equality is obtained by multiplying the equation $(4.2)$ by $\chi_1$ and integrating the resulting equation over $\mathbb{T}^{d+1}$.

 Consequently, $v_0$ satisfies the homogenized equation:
\begin{equation}\left\{\begin{aligned}
\partial_t {v}_0-\Delta {v}_0 +\left(N-\mathcal{M}(\chi_1W)\right){v}_0&=\tilde{f}, \quad\quad\text{in }\Omega\times (0,T),\\
{v}_0&=0, \quad\quad\text{on }\partial\Omega\times (0,T),\\
{v}_0&=g, \quad\quad\text{on }\Omega\times \{t=0\},
\end{aligned}\right.\end{equation}
which yields the first part of Theorem 3.1 after noting that $v_0=u_0 e^{-Nt}$.\\

To see the second part,
by setting
$w_\varepsilon= v_\varepsilon -\varepsilon\chi_1^\varepsilon v_\varepsilon-v_0$, it is easy to check that $w_\varepsilon$ satisfies the following parabolic equation:
\begin{equation}\left\{\begin{aligned}
\partial_t w_\varepsilon-\Delta w_\varepsilon +Nw_\varepsilon&=f_{11}(v_\varepsilon),&\quad &\text{in }\Omega\times (0,T),\\
w_\varepsilon&=0,& &\text{on }\partial\Omega\times (0,T),\\
w_\varepsilon&=-\varepsilon\chi_1^\varepsilon g,& &\text{on }\Omega\times \{t=0\}.
\end{aligned}\right.\end{equation}
with
$$\begin{aligned}f_{11}(v_\varepsilon)=&2\nabla v_\varepsilon\nabla_y\chi_1^\varepsilon-\mathcal{M}(\chi_1W){v}_0 -\chi_1^\varepsilon W^\varepsilon v_\varepsilon-\varepsilon\chi_1^\varepsilon \tilde{f}\\
=&2\nabla w_\varepsilon \nabla_y\chi_1^\varepsilon+2\nabla v_0 \nabla_y\chi_1^\varepsilon+2\varepsilon \chi_1^\varepsilon \nabla v_\varepsilon \nabla_y\chi_1^\varepsilon-\varepsilon\chi_1^\varepsilon \tilde{f}\\
&+ 2v_\varepsilon|\nabla_y\chi_1^\varepsilon|^2
-\mathcal{M}(\chi_1W){v}_0 -\chi_1^\varepsilon W^\varepsilon v_\varepsilon.\end{aligned}$$

\noindent
It is easy to see that
\begin{equation}
2\int_{\Omega_T}\nabla v_0 \nabla_y\chi_1^\varepsilon\cdot w_\varepsilon\leq C\varepsilon ||\Delta v_0||_{L^2(\Omega_T)}||w_\varepsilon||_{L^2(\Omega_T)}+C\varepsilon ||\nabla v_0||_{L^2(\Omega_T)}|| \nabla w_\varepsilon||_{L^2(\Omega_T)},
\end{equation}
and \begin{equation}
2\int_{\Omega_T} \nabla w_\varepsilon \nabla_y\chi_1^\varepsilon \cdot w_\varepsilon\leq \frac1 {16}|| \nabla w_\varepsilon||_{L^2(\Omega_T)}^2+16 ||w_\varepsilon||_{L^2(\Omega_T)}^2.
\end{equation}
Moreover, a direct computation shows that
\begin{equation}\begin{aligned}
&2v_\varepsilon|\nabla_y\chi_1^\varepsilon|^2 -\mathcal{M}(\chi_1W){v}_0 -\chi_1^\varepsilon W^\varepsilon v_\varepsilon\\
=&v_0 \left(2|\nabla_y\chi_1^\varepsilon|^2
-\mathcal{M}(\chi_1W)-\chi_1^\varepsilon W^\varepsilon\right)+(v_\varepsilon-v_0)\left(2|\nabla_y\chi_1^\varepsilon|^2
-\chi_1^\varepsilon W^\varepsilon\right)\\
=&:v_0 \left(2|\nabla_y\chi_1^\varepsilon|^2
-\mathcal{M}(\chi_1W)-\chi_1^\varepsilon W^\varepsilon\right)+E_3,
\end{aligned}\end{equation}
with
\begin{equation}
|E_3|\leq C|w_\varepsilon|+C\varepsilon|v_\varepsilon|.
\end{equation}
To proceed, due to $\mathcal{M}(|\nabla_y\chi_1|^2)=\mathcal{M}(\chi_1W)$, we can introduce $\chi_{1-1}(y,\tau)$ satisfying
\begin{equation}\left\{\begin{aligned}
&\partial_\tau\chi_{1-1} -\Delta_y\chi_{1-1}=2|\nabla_y\chi_1|^2
-\mathcal{M}(\chi_1W)-\chi_1 W\text{ in }\mathbb{T}^{d+1},\\
&\chi_{1-1}\text{ is 1-periodic in }(y,\tau)\text{ with }\mathcal{M}(\chi_{1-1})=0.
\end{aligned}\right.\end{equation}
Note that $\nabla_y\chi_1\in L^\infty(\mathbb{T}^{d+1})$, then according to the parabolic theory,
$$\chi_{1-1}\in W^{2,1}_p(\mathbb{T}^{d+1}),\text{ for any }1<p<\infty,$$
and by embedding,
$$\chi_{1-1} \in C^{1+\alpha,\frac{1+\alpha}2}(\mathbb{T}^{d+1}), \text{ for any }0<\alpha<1.$$
Then, we have
\begin{equation}\begin{aligned}
E_4&=:\int_{\Omega_T}v_0 \left(2|\nabla_y\chi_1^\varepsilon|^2
-\mathcal{M}(\chi_1W)-\chi_1^\varepsilon W^\varepsilon\right)w_\varepsilon\\
&=\varepsilon^2\int_{\Omega_T}\left(\partial_t\chi_{1-1}^\varepsilon-\Delta_x\chi_{1-1}^\varepsilon\right)v_0 w_\varepsilon,
\end{aligned}\end{equation}
which implies that
\begin{equation}\begin{aligned}
|E_4|\leq &C\varepsilon^2||v_0||_{L^\infty_TL^2_x}||w_\varepsilon||_{L^\infty_TL^2_x}+
C\varepsilon^2||\partial_t v_0||_{L^2(\Omega_T))}||w_\varepsilon||_{L^2(\Omega_T))}\\
&+C\varepsilon^2|| v_0||_{L^2(\Omega_T))}||\partial_t w_\varepsilon||_{L^2(\Omega_T))}
+C\varepsilon ||v_0||_{L^2_TH^1_x}||w_\varepsilon||_{L^2_TH^1_x}.
\end{aligned}\end{equation}

Consequently, combining $(4.13)$-$(4.19)$ and $||\partial_t w_\varepsilon||_{L^2(\Omega_T)}\leq C\left(||f||_{L^2(\Omega_T)}+||g||_{H^1(\Omega)}\right)$ after choosing $N$ suitable large yields that
\begin{equation*}||w_\varepsilon||_{L^\infty_TL^2_x} +||\nabla w_\varepsilon||_{L^2(\Omega_T)} +|| w_\varepsilon||_{L^2(\Omega_T)} \leq C\varepsilon \left(||f||_{L^2(\Omega_T)} +||g||_{H^1(\Omega)} \right),\end{equation*}
which yields the second part of Theorem 3.1 after noting that $v_\varepsilon=u_\varepsilon e^{-Nt}$ and $v_0=u_0 e^{-Nt}$.
\qed

\subsection{Proof of Theorem 3.2}
In this subsection, we always assume that $k>2$, $\gamma=1$, $W\in L^\infty(\mathbb{T}^{d+1})$ with $\mathcal{M}(W)=0$ as well  as $\nabla_y\left(W-\mathcal{M}_\tau(W)\right)\in L^\infty(\mathbb{T}^{d+1})$. Inspired by the two scale expansions, we rewrite
\begin{equation*}
\varepsilon^{-1}W^\varepsilon v_\varepsilon =\varepsilon^{-1}\left(W^\varepsilon-\mathcal{M}_\tau(W)(x/\varepsilon)\right)v_\varepsilon
+\varepsilon^{-1}\mathcal{M}_\tau(W)(x/\varepsilon)v_\varepsilon=:\varepsilon^{-1}{W}_1^\varepsilon v_\varepsilon+\varepsilon^{-1}{W}_2^\varepsilon v_\varepsilon.
\end{equation*}

Then, \begin{equation}\mathcal{M}_\tau(W_1)(y)=0\text{ for every }y\in \mathbb{T}^d,\ \mathcal{M}_y(W_2)=\mathcal{M}(W)=0.\end{equation}

In order to prove the homogenization part, we need first to obtain the uniform estimates for $u_\varepsilon$, which is stated in the following lemma.

\begin{lemma}[Uniform estimates]
Under the conditions in Theorem 3.2, there hold the following uniform estimates for $u_\varepsilon$:
\begin{equation}
||u_\varepsilon||_{L^\infty_TL^2_x}+||\nabla u_\varepsilon||_{L^2(\Omega_T)}+|| u_\varepsilon||_{L^2(\Omega_T)}\leq C\left(||g||_{L^2(\Omega)}+||f||_{L^2(\Omega)}\right),
\end{equation}
for a positive constant $C$ depending only on $W$, $d$ and $T$.
\end{lemma}
\begin{proof}
For every $y\in \mathbb{T}^{d}$, we introduce the corrector $\chi_y(y,\tau)$ such that
\begin{equation}
\chi_y(y,\tau)=\int_0^\tau W_1(y,s)ds.
\end{equation}
Then $\chi_y,\nabla_y \chi_y\in L^\infty(\mathbb{T}^{d+1})$ due to $\nabla_y W_1\in L^\infty(\mathbb{T}^{d+1})$, and  it is easy to see that
\begin{equation*}\begin{aligned}
\int_{\Omega_T}\varepsilon^{-1}W_1^\varepsilon v_\varepsilon^2&=\varepsilon^{k-1} \int_{\Omega_T} \partial_t\chi_y^\varepsilon\cdot v_\varepsilon^2=\varepsilon^{k-1}\int_{\Omega\times\{T\}} \chi_y^\varepsilon v_\varepsilon^2- 2\varepsilon^{k-1}\int_{\Omega_T} \chi_y^\varepsilon v_\varepsilon \partial_t v_\varepsilon\\
&=\varepsilon^{k-1}\int_{\Omega\times\{T\}} \chi_y^\varepsilon v_\varepsilon^2-2\varepsilon^{k-1}\int_{\Omega_T} \chi_y^\varepsilon v_\varepsilon\left(\Delta v_\varepsilon-N v_\varepsilon+\varepsilon^{-1}W^\varepsilon v_\varepsilon+\tilde{f}\right),
\end{aligned}\end{equation*}
which, after a direct computation, yields that
\begin{equation}\begin{aligned}
\left|\int_{\Omega_T}\varepsilon^{-1}W_1^\varepsilon v_\varepsilon^2\right|
\leq &C\varepsilon^{k-1}\int_{\Omega\times\{T\}} v_\varepsilon^2
+C\varepsilon^{k-1}\int_{\Omega_T} |\nabla v_\varepsilon|^2
+C\varepsilon^{k-2}\int_{\Omega_T} |\nabla v_\varepsilon|\cdot|v_\varepsilon|\\
&+CN\varepsilon^{k-2}\int_{\Omega_T} |v_\varepsilon|^2+C\varepsilon^{k-1}\int_{\Omega_T} |v_\varepsilon|\cdot |\tilde{f}|.
\end{aligned}\end{equation}
To proceed, let $\chi_2(y)\in W^{2,p}(\mathbb{T}^{d})$, with any $1<p<\infty$, be defined as
\begin{equation}\left\{\begin{aligned}
\Delta_y\chi_{2}=\mathcal{M}_\tau(W)(y)=W_2\  \text{ in }\ \mathbb{T}^{d},\\
\chi_{2}\text{ is 1-periodic with }\mathcal{M}_y(\chi_{2})=0,
\end{aligned}\right.\end{equation}

Then it is easy to see that
\begin{equation*}\begin{aligned}
\varepsilon^{-1}W_2^\varepsilon v_\varepsilon=\varepsilon \Delta_x\chi_2^\varepsilon \cdot v_\varepsilon=\operatorname{div}(v_\varepsilon \nabla_y \chi_2^\varepsilon)-\nabla v_\varepsilon \nabla_y \chi_2^\varepsilon,
\end{aligned}\end{equation*}
which implies that
\begin{equation}\begin{aligned}
\left|\int_{\Omega_T}\varepsilon^{-1}W_2^\varepsilon v_\varepsilon^2\right|\leq \frac{1}{16}||\nabla v_\varepsilon||_{L^2(\Omega_T)}^2+C ||v_\varepsilon||_{L^2(\Omega_T)}^2.
\end{aligned}\end{equation}
Recall that $v_\varepsilon$ satisfies
\begin{equation}\left\{\begin{aligned}
\partial_t v_\varepsilon-\Delta v_\varepsilon+\left(N-\varepsilon^{-1}W^\varepsilon\right)v_\varepsilon&=\tilde{f}, \quad\quad\text{in }\Omega\times (0,T),\\
v_\varepsilon&=0, \quad\quad\text{on }\partial\Omega\times (0,T),\\
v_\varepsilon&=g, \quad\quad\text{on }\Omega\times \{t=0\}.
\end{aligned}\right.\end{equation}

Therefore, choosing $N$ suitable large and $\varepsilon\leq \varepsilon_0$ with $\varepsilon_0$ depending only on $W$ and $d$, and   combining $(4.23)$-$(4.26)$, we have
\begin{equation}\begin{aligned}
||v_\varepsilon||_{L^\infty_TL^2_x}+||v_\varepsilon||_{L^2(\Omega_T)}+||\nabla v_\varepsilon||_{L^2(\Omega_T)}\leq C\left(||g||_{L^2(\Omega)}+||f||_{L^2(\Omega)}\right),
\end{aligned}\end{equation}
for a positive constant $C$ depending only on $||W||_{L^\infty(\mathbb{T}^{d+1})}$ and $T$, which yields the desired estimates $(4.21)$ after noting that $v_\varepsilon=u_\varepsilon e^{-Nt}$, and the uniform estimates $(4.21)$ follows directly from Lemma 4.1 if $\varepsilon>\varepsilon_0$.\end{proof}

\noindent\textbf{Proof of Theorem 3.2: the homogenization equation}.

With the same notations in Lemma 4.5, we have
\begin{equation}\begin{aligned}
&\varepsilon^{-1}W_1^\varepsilon v_\varepsilon=\varepsilon^{k-1} v_\varepsilon \partial_t\chi_y^\varepsilon=
\varepsilon^{k-1}\partial_t\left(v_\varepsilon\chi_y^\varepsilon\right)-\varepsilon^{k-1}\chi_y^\varepsilon \partial_t v_\varepsilon\\
=&\varepsilon^{k-1}\partial_t\left(v_\varepsilon\chi_y^\varepsilon\right)
-\varepsilon^{k-1}\chi_y^\varepsilon\left(\Delta v_\varepsilon-N v_\varepsilon+\varepsilon^{-1}W^\varepsilon v_\varepsilon+\tilde{f}\right)\\
=&\varepsilon^{k-1}\partial_t\left(v_\varepsilon\chi_y^\varepsilon\right)-
\varepsilon^{k-1}\Delta_x\left(v_\varepsilon\chi_y^\varepsilon\right)+\varepsilon^{k-2}\nabla_y \chi_y^\varepsilon \nabla v_\varepsilon\\
&+\varepsilon^{k-2}\operatorname{div}_x\left(v_\varepsilon \nabla_y \chi_y^\varepsilon\right)+\varepsilon^{k-1}\chi_y^\varepsilon\left(N v_\varepsilon-\varepsilon^{-1}W^\varepsilon v_\varepsilon-\tilde{f}\right),
\end{aligned}\end{equation}
and

\begin{equation}\begin{aligned}
&\varepsilon^{-1}{W}_2^\varepsilon v_\varepsilon=\varepsilon v_\varepsilon\Delta_x\chi_{2}^\varepsilon=
\varepsilon \Delta_x\left(v_\varepsilon\chi_{2}^\varepsilon\right)-2\nabla v_\varepsilon \nabla_y \chi_{2}^\varepsilon-\varepsilon \chi_{2}^\varepsilon \Delta_x v_\varepsilon\\
=&\varepsilon \Delta_x\left(v_\varepsilon\chi_{2}^\varepsilon\right)-2\nabla v_\varepsilon \nabla_y \chi_{2}^\varepsilon -\varepsilon \chi_{2}^\varepsilon\left(\partial_t v_\varepsilon+N v_\varepsilon-\varepsilon^{-1}W^\varepsilon v_\varepsilon-\tilde{f}\right)\\
=&\varepsilon \Delta_x\left(v_\varepsilon\chi_{2}^\varepsilon\right)-\varepsilon \partial_t\left(v_\varepsilon\chi_{2}^\varepsilon\right)-2\nabla v_\varepsilon \nabla_y\chi_{2}^\varepsilon
+\chi_{2}^\varepsilon W^\varepsilon v_\varepsilon -\varepsilon\chi_{2}^\varepsilon\left(N v_\varepsilon-\tilde{f}\right).
\end{aligned}\end{equation}

Denote $$\tilde{v}_\varepsilon =v_\varepsilon-\varepsilon^{k-1}\chi_y^\varepsilon v_\varepsilon+\varepsilon\chi_{2}^\varepsilon v_\varepsilon,$$ then $\tilde{v}_\varepsilon$ satisfies
\begin{equation}\left\{\begin{aligned}
\partial_t \tilde{v}_\varepsilon-\Delta \tilde{v}_\varepsilon+N\tilde{v}_\varepsilon&=f_{21}({v}_\varepsilon),& &\text{in }\Omega\times (0,T),\\
\tilde{v}_\varepsilon&=0,& &\text{on }\partial\Omega\times (0,T),\\
\tilde{v}_\varepsilon&=(1+\varepsilon\chi_2^\varepsilon)g,&\quad &\text{on }\Omega\times \{t=0\}.
\end{aligned}\right.\end{equation}
with $$\begin{aligned}f_{21}({v}_\varepsilon)=&\tilde{f}(1-\varepsilon^{k-1}\chi_y^\varepsilon+\varepsilon\chi_{2}^\varepsilon)-2\nabla v_\varepsilon \nabla_y\chi_{2}^\varepsilon
+\chi_{2}^\varepsilon W^\varepsilon v_\varepsilon
+\varepsilon^{k-2}\nabla_y \chi_y^\varepsilon \nabla v_\varepsilon\\
&+\varepsilon^{k-2}\operatorname{div}_x\left(v_\varepsilon \nabla_y \chi_y^\varepsilon\right)-\varepsilon^{k-2}\chi_y^\varepsilon W^\varepsilon v_\varepsilon\\
=&\tilde{f}(1-\varepsilon^{k-1}\chi_y^\varepsilon+\varepsilon\chi_{2}^\varepsilon)-2\nabla \tilde{v}_\varepsilon \nabla_y\chi_{2}^\varepsilon+2 \varepsilon\chi_{2}^\varepsilon\nabla v_\varepsilon \nabla_y\chi_{2}^\varepsilon+2v_\varepsilon |\nabla_y\chi_{2}^\varepsilon|^2+\chi_{2}^\varepsilon W^\varepsilon v_\varepsilon\\
&-\varepsilon^{k-1}\chi_y^\varepsilon\nabla v_\varepsilon\nabla_y\chi_{2}^\varepsilon-\varepsilon^{k-2}v_\varepsilon\nabla_y \chi_y^\varepsilon \nabla_y\chi_{2}^\varepsilon+\varepsilon^{k-2}\nabla_y \chi_y^\varepsilon \nabla v_\varepsilon\\
&\quad+\varepsilon^{k-2}\operatorname{div}_x\left(v_\varepsilon \nabla_y \chi_y^\varepsilon\right)-\varepsilon^{k-2}\chi_y^\varepsilon W^\varepsilon v_\varepsilon.\end{aligned}$$

\noindent Applying Lemma 4.1 after in view of Lemma 4.5, we have
\begin{equation}||\tilde{v}_\varepsilon||_{L^2(0,T;H^{1}_0(\Omega))}\leq C\left(||f||_{L^2(\Omega_T)}+||g||_{L^2(\Omega)}\right).\end{equation}
Moreover, it is easy to see that
$$||\partial_t \tilde{v}_\varepsilon||_{L^2(0,T;H^{-1}(\Omega))}\leq C\left(||f||_{L^2(\Omega_T)}+||g||_{L^2(\Omega)}\right),$$
which combines with $(4.31)$ and Lemma 4.3 implies that there exists $v_0\in L^2(0,T;H^{1}_0(\Omega))$, such that
\begin{equation}\begin{aligned}
\tilde{v}_\varepsilon \rightharpoonup v_0\text{ weakly in }L^2(0,T;H^{1}_0(\Omega));\\
\tilde{v}_\varepsilon \rightarrow v_0\text{ strongly in }L^2(\Omega_T);\\
v_\varepsilon \rightarrow v_0\text{ strongly in }L^2(\Omega_T).
\end{aligned}\end{equation}
To proceed, we note that \begin{equation}0=\mathcal{M}(|\nabla_y\chi_{2}|^2+\chi_{2} W_2)=\mathcal{M}(|\nabla_y\chi_{2}|^2+\chi_{2} W),\end{equation}
where the first equality is due to multiplying the equation $(4.24)$ by $\chi_{2}$ and integrating the resulting equation over $\mathbb{T}^{d},$
and the second equality follows from
$$\begin{aligned}
&\mathcal{M}(\chi_{2} W)=\mathcal{M}(\chi_{2} W_2)+\mathcal{M}(\chi_{2} W_1)\\
=&\mathcal{M}_y(\chi_{2} W_2)+\mathcal{M}_y(\chi_{2}\cdot \mathcal{M}_\tau (W_1))
=\mathcal{M}(\chi_{2} W_2),
\end{aligned}$$

To see the homogenization limit after noting that $\chi_2=\chi_2(y)$ independent of $\tau$, we rewrite the term $-2\nabla \tilde{v}_\varepsilon \nabla_y\chi_{2}^\varepsilon$ as below:
\begin{equation}\begin{aligned}
&-2\nabla \tilde{v}_\varepsilon \nabla_y\chi_{2}^\varepsilon=-2\varepsilon \operatorname{div}_x\left(\chi_{2}^\varepsilon \nabla \tilde{v}_\varepsilon\right)+2\varepsilon\chi_{2}^\varepsilon \Delta\tilde{v}_\varepsilon\\
=&-2\varepsilon \operatorname{div}_x\left(\chi_{2}^\varepsilon \nabla \tilde{v}_\varepsilon\right)+2\varepsilon\chi_{2}^\varepsilon\left(\partial_t \tilde{v}_\varepsilon+N \tilde{v}_\varepsilon-f_{21}({v}_\varepsilon)\right)\\
=&-2\varepsilon \operatorname{div}_x\left(\chi_{2}^\varepsilon \nabla \tilde{v}_\varepsilon\right)+2\varepsilon\partial_t\left(\chi_{2}^\varepsilon \tilde{v}_\varepsilon\right)+
2\varepsilon\chi_{2}^\varepsilon\left(N \tilde{v}_\varepsilon-f_{21}({v}_\varepsilon)\right)\\
&\quad\quad\rightarrow 0\quad\quad\text{ strongly in }\mathcal{D}'(\Omega_T).
\end{aligned}\end{equation}

Consequently, $v_0$ satisfies the homogenized equation:
\begin{equation}\left\{\begin{aligned}
\partial_t {v}_0-\Delta {v}_0+(N+\mathcal{M}(\chi_{2}W)){v}_0&=\tilde{f}, \quad\quad\text{in }\Omega\times (0,T),\\
{v}_0&=0, \quad\quad\text{on }\partial\Omega\times (0,T),\\
{v}_0&=g, \quad\quad\text{on }\Omega\times \{t=0\},
\end{aligned}\right.\end{equation}
which yields the desired limit equation $(3.4)$ after noting that $v_0=u_0 e^{-Nt}$.
\qed\\

Now we are ready to prove the second part of Theorem 3.2.\\
\noindent \textbf{Prove of Theorem 3.2: the convergence rates}.

By setting $$w_\varepsilon=\tilde{v}_\varepsilon -v_0=v_\varepsilon-\varepsilon^{k-1}v_\varepsilon\chi_y^\varepsilon+\varepsilon\chi_{2}^\varepsilon v_\varepsilon-v_0,$$ $w_\varepsilon$ satisfies the following parabolic equation:

\begin{equation}\left\{\begin{aligned}
\partial_t w_\varepsilon-\Delta w_\varepsilon+Nw_\varepsilon&=f_{22}({v}_\varepsilon),&\quad &\text{in }\Omega\times (0,T),\\
w_\varepsilon&=0,& &\text{on }\partial\Omega\times (0,T),\\
w_\varepsilon&=\varepsilon\chi_2^\varepsilon g,& &\text{on }\Omega\times \{t=0\},
\end{aligned}\right.\end{equation}
with $$\begin{aligned}f_{22}({v}_\varepsilon)=&(\varepsilon\chi_{2}^\varepsilon-\varepsilon^{k-1}\chi_y^\varepsilon)\tilde{f}-2\nabla \tilde{v}_\varepsilon \nabla_y\chi_{2}^\varepsilon+2 \varepsilon\chi_{2}^\varepsilon\nabla v_\varepsilon \nabla_y\chi_{2}^\varepsilon+2v_\varepsilon |\nabla_y\chi_{2}^\varepsilon|^2+\chi_{2}^\varepsilon W^\varepsilon v_\varepsilon\\
&+\mathcal{M}(\chi_{2}W){v}_0-\varepsilon^{k-1}\chi_y^\varepsilon\nabla v_\varepsilon\nabla_y\chi_{2}^\varepsilon-\varepsilon^{k-2}v_\varepsilon\nabla_y \chi_y^\varepsilon \nabla_y\chi_{2}^\varepsilon+\varepsilon^{k-2}\nabla_y \chi_y^\varepsilon \nabla v_\varepsilon\\
&\quad+\varepsilon^{k-2}\operatorname{div}_x\left(v_\varepsilon \nabla_y \chi_y^\varepsilon\right)-\varepsilon^{k-2}\chi_y^\varepsilon W^\varepsilon v_\varepsilon.\end{aligned}$$

First, it is easy to see that
\begin{equation}\begin{aligned}
\int_{\Omega_T}-2\nabla \tilde{v}_\varepsilon \nabla_y\chi_{2}^\varepsilon\cdot w_\varepsilon
\leq& C||w_\varepsilon||_{L^2(\Omega_T)}||\nabla w_\varepsilon||_{L^2(\Omega_T)}-\int_{\Omega_T}2\nabla v_0 \nabla_y\chi_{2}^\varepsilon\cdot w_\varepsilon\\
\leq& C||w_\varepsilon||_{L^2(\Omega_T)}||\nabla w_\varepsilon||_{L^2(\Omega_T)}+C\varepsilon ||\Delta v_0||_{L^2(\Omega_T)}|| w_\varepsilon||_{L^2(\Omega_T)}\\
&\quad+C\varepsilon ||\nabla v_0||_{L^2(\Omega_T)}|| \nabla w_\varepsilon||_{L^2(\Omega_T)}\\
\leq& \frac 1{16}|| \nabla w_\varepsilon||_{L^2(\Omega_T)}^2+C|| w_\varepsilon||_{L^2(\Omega_T)}^2+C\varepsilon^2 ||v_0||_{L^2(0,T;H^2(\Omega))}.
\end{aligned}\end{equation}

Next, due to $\mathcal{M}(|\nabla_y\chi_{2}|^2+\chi_{2} W_2)=\mathcal{M}(|\nabla_y\chi_{2}|^2+\chi_{2} W)=0$, we can introduce the corrector $\chi_{2-1}(y,\tau)$ satisfying
\begin{equation}\left\{\begin{aligned}
\partial_\tau\chi_{2-1}- \Delta_y\chi_{2-1}=2 |\nabla_y\chi_{2}|^2+\chi_{2} W+\mathcal{M}(\chi_{2}W)\text{ in }\mathbb{T}^{d+1},\\
\chi_{2-1}\text{ is 1-periodic in }(y,\tau)\quad\text{ with }\quad\mathcal{M}(\chi_{2-1})(\tau)=0.
\end{aligned}\right.\end{equation}
According to the parabolic theory,
$$\chi_{2-1}\in W^{2,1}_p(\mathbb{T}^{d+1}),\text{ for any }1<p<\infty,$$
and by embedding,
$$\chi_{2-1} \in C^{1+\alpha,\frac{1+\alpha}2}(\mathbb{T}^{d+1}), \text{ for any }0<\alpha<1.$$

Then, we have
\begin{equation}\begin{aligned}
E_5&=:\int_{\Omega_T}\left(2v_\varepsilon |\nabla_y\chi_{2}^\varepsilon|^2+\chi_{2}^\varepsilon W^\varepsilon  v_\varepsilon+\mathcal{M}(\chi_{2}W){v}_0\right)w_\varepsilon\\
&=\int_{\Omega_T}v_0 \left(2|\nabla_y\chi_{2}^\varepsilon|^2+\chi_{2}^\varepsilon W^\varepsilon  +\mathcal{M}(\chi_{2}W_2)\right)w_\varepsilon+\int_{\Omega_T}\left(2|\nabla_y\chi_{2}^\varepsilon|^2+\chi_{2}^\varepsilon W^\varepsilon \right)w_\varepsilon\cdot w_\varepsilon\\
&\quad+\int_{\Omega_T}\left(2|\nabla_y\chi_{2}^\varepsilon|^2+\chi_{2}^\varepsilon W^\varepsilon \right)(\varepsilon^{k-1}\chi_y^\varepsilon-\varepsilon\chi_{2}^\varepsilon) v_\varepsilon\cdot w_\varepsilon\\
&=:E_{51}+E_{52}+E_{53}.
\end{aligned}\end{equation}
It is easy to see that
\begin{equation}
|E_{52}|+|E_{53}|\leq C|| w_\varepsilon||_{L^2(\Omega_T)}^2+C\varepsilon^2 || v_\varepsilon||_{L^2(\Omega_T)}^2.
\end{equation}
For the first term $E_{51}$, a direct computation shows that
\begin{equation}\begin{aligned}E_{51}=&\varepsilon^k\int_{\Omega_T} \partial_t \chi_{2-1}^\varepsilon\cdot v_0 w_\varepsilon- \varepsilon^2\int_{\Omega_T}\Delta_x\chi_{2-1}^\varepsilon\cdot v_0 w_\varepsilon\\
\leq&-\varepsilon^k\int_{\Omega_T}\chi_{2-1}^\varepsilon v_0\partial_tw_\varepsilon
+C\varepsilon || w_\varepsilon||_{L^2(0,T;H^1(\Omega))}|| v_0||_{L^2(0,T;H^1(\Omega))}\\
&+C\varepsilon^k ||w_\varepsilon||_{L^\infty_TL^2_x}||v_0||_{L^\infty_TL^2_x}+C\varepsilon^k||w_\varepsilon||_{L^2(\Omega_T)}||\partial_t v_0||_{L^2(\Omega_T)}.
\end{aligned}\end{equation}
To proceed, we rewrite $$w_\varepsilon=w_{1,\varepsilon}+w_{2,\varepsilon}$$
with $w_{2,\varepsilon}$ defined by
\begin{equation*}\left\{\begin{aligned}
\partial_t w_{2,\varepsilon}-\Delta w_{2,\varepsilon}+Nw_{2,\varepsilon}&=\varepsilon^{k-2}\operatorname{div}_x\left(v_\varepsilon \nabla_y \chi_y^\varepsilon\right),&\quad &\text{in }\Omega\times (0,T),\\
{w}_{2,\varepsilon}&=0,& &\text{on }\partial\Omega\times (0,T),\\
{w}_{2,\varepsilon}&=0,& &\text{on }\Omega\times \{t=0\},
\end{aligned}\right.\end{equation*}
Then, we have $$\begin{aligned}
&||\partial_t w_{1,\varepsilon}||_{L^2(\Omega_T)}\leq C\left(||f||_{L^2(\Omega_T)}+||g||_{H^1(\Omega)}\right) \\
&||\partial_t w_{2,\varepsilon}||_{L^2(0,T;H^{-1}(\Omega))}\leq C\varepsilon^{k-2}\left(||f||_{L^2(\Omega_T)}+||g||_{H^1(\Omega)}\right),\end{aligned}$$
which implies that
\begin{equation}\begin{aligned}
-\varepsilon^k\int_{\Omega_T}\chi_{2-1}^\varepsilon v_0\partial_tw_\varepsilon=&-\varepsilon^k\int_{\Omega_T}\chi_{2-1}^\varepsilon v_0\partial_tw_{1,\varepsilon}-\varepsilon^k\int_{\Omega_T}\chi_{2-1}^\varepsilon v_0\partial_tw_{2,\varepsilon}\\
\leq& C\varepsilon^{k}||\chi_{2-1}^\varepsilon v_0||_{L^2(0,T;H^{1}_0(\Omega))} ||\partial_t w_{2,\varepsilon}||_{L^2(0,T;H^{-1}(\Omega))}
\\&\quad+C\varepsilon^k||v_0||_{L^2(\Omega_T)}||\partial_t w_{1,\varepsilon}||_{L^2(\Omega_T)}\\
\leq& C \max(\varepsilon^{2k-3},\varepsilon^k)\left(||f||_{L^2(\Omega_T)}^2+||g||_{H^1(\Omega)}^2\right).
\end{aligned}\end{equation}

Consequently, multiplying the equation $(4.36)$ by $w_\varepsilon$ after choosing $N$ suitable large and combining $(4.37)$-$(4.42)$ yields that
\begin{equation}||w_\varepsilon||_{L^\infty_TL^2_x}+||\nabla w_\varepsilon||_{L^2(\Omega_T)}+|| w_\varepsilon||_{L^2(\Omega_T)}\leq C\max\left(\varepsilon,\varepsilon^{k-2}\right)\left(||f||_{L^2(\Omega_T)}+||g||_{H^1(\Omega)}\right),\end{equation}
which completes the proof of Theorem 3.2.\qed

\section{Proofs of Theorems 3.3 and  3.4}
We first give the proof of Theorem 3.3.
\subsection{Proof of Theorem 3.3}
In this subsection, we give the proof of Theorem 3.3 and we always assume $1<k<2$, $\gamma=1$, $W\in L^\infty(\mathbb{T}^{d+1})$ and $\partial_\tau\left(W-\mathcal{M}_y(W)\right)\in L^\infty(\mathbb{T}^{d+1})$ with $\mathcal{M}(W)=0$. First we deduce a uniform estimates for $u_\varepsilon$.
\begin{lemma}[Uniform estimates]
Under the conditions in Theorem 3.3, there hold the following uniform estimates for $u_\varepsilon$:
\begin{equation}
||u_\varepsilon||_{L^\infty_TL^2_x}+||\nabla u_\varepsilon||_{L^2(\Omega_T)}+|| u_\varepsilon||_{L^2(\Omega_T)}\leq C\left(||g||_{L^2(\Omega)}+||f||_{L^2(\Omega)}\right),
\end{equation}
for a positive constant $C$ depending only on $W$, $d$, $T$ and $k$. Moreover, the constant $C$ will blow up as $k\rightarrow 1$ for $k\in(1,2)$.
\end{lemma}
\begin{proof}
First, inspired by the two-scale expansions, we rewrite the term $\varepsilon^{-1}W^\varepsilon v_\varepsilon$ as
\begin{equation}\varepsilon^{-1}W^\varepsilon v_\varepsilon=\varepsilon^{-1}(W^\varepsilon-\mathcal{M}_y(W)(t/\varepsilon^k)) v_\varepsilon+\varepsilon^{-1}\mathcal{M}_y(W)(t/\varepsilon^k) v_\varepsilon=:\varepsilon^{-1}W^\varepsilon_3v_\varepsilon+\varepsilon^{-1}W^\varepsilon_4v_\varepsilon.
\end{equation}
And note that
\begin{equation*}\begin{aligned}
\mathcal{M}_y(W_3)(\tau)=0\text{ for every }\tau\in \mathbb{T};\ \mathcal{M}_\tau(W_4)=\mathcal{M}(W)=0.
\end{aligned}\end{equation*}

Recall that, for every $\tau\in \mathbb{T}$, the corrector $\chi_3(y,\tau)$ is defined as,
\begin{equation}\left\{\begin{aligned}
\Delta_y \chi_3=W-\mathcal{M}_y(W)=W_3\text{ in }\mathbb{T}^{d},\\
\chi_3\text{ is 1-periodic},\ \mathcal{M}_y(\chi_3)(\tau)=0.
\end{aligned}\right.\end{equation}
By elliptic theory, we have
\begin{equation*}
\chi_{3}\in L^\infty(\mathbb{T};W^{2,p}(\mathbb{T}^{d})),\text{ for any }1<p<\infty
\end{equation*}
and by embedding,
$$\chi_{3} \in L^\infty(\mathbb{T};C^{1+\alpha}(\mathbb{T}^{d})), \text{ for any }0<\alpha<1.$$
Moreover, due to $\partial_\tau\left(W-\mathcal{M}_y(W)\right)\in L^\infty(\mathbb{T}^{d+1})$, we know that
\begin{equation}\begin{aligned}
&\partial_\tau\chi_{3}\in L^\infty(\mathbb{T};W^{2,p}(\mathbb{T}^{d})),\text{ for any }1<p<\infty;\\
&\partial_\tau\chi_{3} \in L^\infty(\mathbb{T};C^{1+\alpha}(\mathbb{T}^{d})), \text{ for any }0<\alpha<1.
\end{aligned}\end{equation}

Then it is easy to see that
\begin{equation}\begin{aligned}
\varepsilon^{-1}W^\varepsilon_3 v_\varepsilon=\varepsilon \Delta_x\chi_3^\varepsilon \cdot v_\varepsilon=\operatorname{div}(v_\varepsilon \nabla_y \chi_3^\varepsilon)-\nabla v_\varepsilon \nabla_y \chi_3^\varepsilon.
\end{aligned}\end{equation}
For the term $\varepsilon^{-1}W^\varepsilon_4v_\varepsilon$, we introduce the corrector $\chi_{3-1}(\tau)$
defined as
\begin{equation}\chi_{3-1}(\tau)=\int_0^\tau W_4(s)ds,\end{equation}
which implies
$$||\chi_{3-1}||_{L^\infty(\mathbb{T})}\leq ||W||_{L^\infty(\mathbb{T}^{d+1})}.$$
Moreover, due to  $$\mathcal{M}_\tau(\chi_{3-1}W_4)=\frac{1}{2}\left[\mathcal{M}_\tau(W_4)\right]^2=0,$$
we can introduce $\chi_{3-2}(\tau)\in L^\infty(\mathbb{T})$ defined as
$$\chi_{3-2}(\tau)=\int_0^\tau\left(\chi_{3-1}W_4\right)(s)ds,$$
with $$||\chi_{3-2}||_{L^\infty(\mathbb{T})}\leq ||W||^2_{L^\infty(\mathbb{T}^{d+1})}.$$
And similarly, we have
$$\mathcal{M}_\tau(\chi_{3-2}W_4)=\frac{1}{6}\left[\mathcal{M}_\tau(W_4)\right]^3=0,$$
then we can introduce $\chi_{3-3}(\tau)\in L^\infty(\mathbb{T})$ defined as
$$\chi_{3-3}(\tau)=\int_0^\tau\left(\chi_{3-2}W_4\right)(s)ds,$$
with $$||\chi_{3-3}||_{L^\infty(\mathbb{T})}\leq ||W||^3_{L^\infty(\mathbb{T}^{d+1})}.$$

By induction, for $i=2,3,\cdots$, we can introduce
$\chi_{3-i}(\tau)\in L^\infty(\mathbb{T})$ defined as

\begin{equation}\chi_{3-i}(\tau)=\int_0^\tau\left(\chi_{3-(i-1)}W_4\right)(s)ds,\end{equation}
with $\chi_{3-1}$ defined in $(5.6)$,
 $$||\chi_{3-i}||_{L^\infty(\mathbb{T})}\leq ||W||^i_{L^\infty(\mathbb{T}^{d+1})}$$
 and
\begin{equation}\mathcal{M}_\tau(\chi_{3-i}W_4)=\frac{1}{(i+1)!}\left[\mathcal{M}_\tau(W_4)\right]^{i+1}=0.\end{equation}

Then a direct computation shows that
\begin{equation}\begin{aligned}
&\varepsilon^{-1}W_4^\varepsilon v_\varepsilon=\varepsilon^{k-1}v_\varepsilon\partial_t\chi_{3-1}^\varepsilon =\varepsilon^{k-1}\partial_t (v_\varepsilon\chi_{3-1}^\varepsilon)-\varepsilon^{k-1}\chi_{3-1}^\varepsilon\partial_t v_\varepsilon\\
=&\varepsilon^{k-1}\partial_t (v_\varepsilon\chi_{3-1}^\varepsilon)-\varepsilon^{k-1}\chi_{3-1}^\varepsilon\left(\Delta v_\varepsilon +\varepsilon^{-1}W^\varepsilon v_\varepsilon-Nv_\varepsilon +\tilde{f}\right)\\
=&\varepsilon^{k-1}\partial_t (v_\varepsilon\chi_{3-1}^\varepsilon)-\varepsilon^{k-1}\Delta_x (v_\varepsilon\chi_{3-1}^\varepsilon)-\varepsilon^{k-2}\chi_{3-1}^\varepsilon W_3^\varepsilon v_\varepsilon-\varepsilon^{k-1}\chi_{3-1}^\varepsilon \tilde{f}\\
&\quad+N\varepsilon^{k-1}\chi_{3-1}^\varepsilon v_\varepsilon-\varepsilon^{k-2}\chi_{3-1}^\varepsilon W_4^\varepsilon v_\varepsilon.
\end{aligned}\end{equation}
According to $(5.3)$, we have
\begin{equation}\begin{aligned}
-\varepsilon^{k-2}\chi_{3-1}^\varepsilon W_3^\varepsilon v_\varepsilon&=-\varepsilon^{k}\chi_{3-1}^\varepsilon \Delta_x\chi_3^\varepsilon \cdot v_\varepsilon=-\varepsilon^{k}\Delta_x\left(\chi_{3-1}^\varepsilon\chi_3^\varepsilon\right)\cdot v_\varepsilon\\
&=-\varepsilon^{k-1}\operatorname{div}_x\left(\chi_{3-1}^\varepsilon v_\varepsilon\nabla_y\chi_3^\varepsilon\right)+\varepsilon^{k-1}\chi_{3-1}^\varepsilon\nabla_y\chi_3^\varepsilon \nabla v_\varepsilon.
\end{aligned}\end{equation}
and due to $(5.7)$, we have
\begin{equation}\begin{aligned}
&-\varepsilon^{k-2}\chi_{3-1}^\varepsilon W_4^\varepsilon v_\varepsilon=-\varepsilon^{2k-2}v_\varepsilon \partial_t \chi_{3-2}^\varepsilon
=-\varepsilon^{2k-2}\partial_t(v_\varepsilon  \chi_{3-2}^\varepsilon)+\varepsilon^{2k-2} \chi_{3-2}^\varepsilon\partial_tv_\varepsilon\\
=&-\varepsilon^{2k-2}\partial_t(v_\varepsilon  \chi_{3-2}^\varepsilon)+\varepsilon^{2k-2} \chi_{3-2}^\varepsilon\left(\Delta v_\varepsilon +\varepsilon^{-1}W^\varepsilon v_\varepsilon-Nv_\varepsilon +\tilde{f}\right)\\
=&-\varepsilon^{2k-2}\partial_t(v_\varepsilon  \chi_{3-2}^\varepsilon)+\varepsilon^{2k-2}\Delta(v_\varepsilon  \chi_{3-2}^\varepsilon)+\varepsilon^{2k-3}\chi_{3-2}^\varepsilon W_3^\varepsilon v_\varepsilon\\
&\quad\quad-\varepsilon^{2k-2} \chi_{3-2}^\varepsilon\left(Nv_\varepsilon -\tilde{f}\right)+\varepsilon^{2k-3}\chi_{3-2}^\varepsilon W_4^\varepsilon v_\varepsilon.
\end{aligned}\end{equation}
By induction, for $i\in \mathbb{N}_+$ with $i\geq 2$ we have

\begin{equation}\begin{aligned}
&(-1)^i\varepsilon^{i(k-1)-1}\chi_{3-i}^\varepsilon W_3^\varepsilon v_\varepsilon\\
=&(-1)^i\varepsilon^{i(k-1)+1}\chi_{3-i}^\varepsilon v_\varepsilon\Delta_x\chi_3^\varepsilon =(-1)^i\varepsilon^{i(k-1)+1}\Delta_x\left(\chi_{3-i}^\varepsilon\chi_3^\varepsilon\right)\cdot v_\varepsilon\\
=&(-1)^i\varepsilon^{i(k-1)}\operatorname{div}_x\left(\chi_{3-i}^\varepsilon v_\varepsilon\nabla_y\chi_3^\varepsilon\right)+(-1)^{i+1}\varepsilon^{i(k-1)}\chi_{3-i}^\varepsilon\nabla_y\chi_3^\varepsilon \nabla v_\varepsilon,
\end{aligned}\end{equation}
and
\begin{equation}\begin{aligned}
&(-1)^i\varepsilon^{i(k-1)-1}\chi_{3-i}^\varepsilon W_4^\varepsilon v_\varepsilon
=(-1)^i\varepsilon^{(i+1)(k-1)}v_\varepsilon \partial_t \chi_{3-(i+1)}^\varepsilon\\
=&(-1)^i\varepsilon^{(i+1)(k-1)}\partial_t(v_\varepsilon  \chi_{3-(i+1)}^\varepsilon)+(-1)^{i+1}\varepsilon^{(i+1)(k-1)} \chi_{3-(i+1)}^\varepsilon\partial_tv_\varepsilon\\
=&(-1)^i\varepsilon^{(i+1)(k-1)}\partial_t(v_\varepsilon  \chi_{3-(i+1)}^\varepsilon)+(-1)^{i+1}\varepsilon^{(i+1)(k-1)} \chi_{3-(i+1)}^\varepsilon\left(\Delta v_\varepsilon +\varepsilon^{-1}W^\varepsilon v_\varepsilon-Nv_\varepsilon +\tilde{f}\right)\\
=&(-1)^i\varepsilon^{(i+1)(k-1)}\partial_t(v_\varepsilon  \chi_{3-(i+1)}^\varepsilon)+(-1)^{i+1}\varepsilon^{(i+1)(k-1)}\Delta(v_\varepsilon  \chi_{3-(i+1)}^\varepsilon)\\
&\quad+(-1)^{i+1}\varepsilon^{(i+1)(k-1)-1}\chi_{3-(i+1)}^\varepsilon W_3^\varepsilon v_\varepsilon
+(-1)^{i}\varepsilon^{(i+1)(k-1)} \chi_{3-(i+1)}^\varepsilon\left(Nv_\varepsilon -\tilde{f}\right)\\
&\quad\quad\quad+(-1)^{i+1}\varepsilon^{(i+1)(k-1)-1}\chi_{3-(i+1)}^\varepsilon W_4^\varepsilon v_\varepsilon.
\end{aligned}\end{equation}

In conclusion, we define
\begin{equation}
\tilde{v}_\varepsilon=:v_\varepsilon+\sum_{i=1}^{I_k}(-1)^i\varepsilon^{i(k-1)}\chi_{3-i}^\varepsilon v_\varepsilon,
\end{equation}
where $I_k$ is the smallest positive integer such that
\begin{equation}
I_k(k-1)\geq 1+(k-1)=k.
\end{equation}

Note that $I_k\rightarrow |\log k|^{-1}$ when $k>1$ and $k \rightarrow 1$. Therefore, combining $(5.9)$-$(5.14)$ yields that $\tilde{v}_\varepsilon$ satisfies the following parabolic equation:
\begin{equation}\left\{\begin{aligned}
\partial_t \tilde{v}_\varepsilon-\Delta \tilde{v}_\varepsilon+N\tilde{v}_\varepsilon&=f_{31}(v_\varepsilon),&\quad &\text{in }\Omega\times (0,T),\\
\tilde{v}_\varepsilon&=0,& &\text{on }\partial\Omega\times (0,T),\\
\tilde{v}_\varepsilon&=g.& &\text{on }\Omega\times \{t=0\},
\end{aligned}\right.\end{equation}
with
\begin{equation*}\begin{aligned}
f_{31}(v_\varepsilon)=&:\operatorname{div}(v_\varepsilon \nabla_y \chi_3^\varepsilon)-\nabla v_\varepsilon \nabla_y\chi_3^\varepsilon+\sum_{i=1}^{I_k}(-1)^i\varepsilon^{i(k-1)}\operatorname{div}_x\left(\chi_{3-i}^\varepsilon v_\varepsilon\nabla_y\chi_3^\varepsilon\right)\\
&\quad+\sum_{i=1}^{I_k}(-1)^{i+1}\varepsilon^{i(k-1)}\chi_{3-i}^\varepsilon\nabla_y\chi_3^\varepsilon \nabla v_\varepsilon+\tilde{f}+\sum_{i=1}^{I_k}(-1)^i\varepsilon^{i(k-1)}\chi_{3-i}^\varepsilon\tilde{f}\\
&\quad\quad+(-1)^{I_k+1}\varepsilon^{(I_k+1)(k-1)-1}\chi_{3-(I_k+1)}^\varepsilon W_4^\varepsilon v_\varepsilon.
\end{aligned}\end{equation*}

Since $1<k<2$ is a fixed positive constant, we can choose $\varepsilon$ small enough depending only on $k-1$ as well as $W$ and the universal constant $N$ suitably large, such that the following energy estimates hold true:
\begin{equation}
||\tilde{v}_\varepsilon||_{L^\infty_TL^2_x}+||\nabla \tilde{v}_\varepsilon||_{L^2(\Omega_T)}+|| \tilde{v}_\varepsilon||_{L^2(\Omega_T)}\leq C(||g||_{L^2(\Omega)}+||f||_{L^2(\Omega_T)})+\frac14 ||v_\varepsilon||_{L^2_T H^1_x}.
\end{equation}

Note that, for any $i\in\mathbb{N}_+$, the corrector $\chi_{3-i}(\tau)$ is independent of $y$, then for $k\in(1,2)$ fixed, we can choose $\varepsilon$ small enough such that
\begin{equation}
||v_\varepsilon||_{L^2_T H^1_x}\leq \frac3 2 ||\tilde{v}_\varepsilon||_{L^2_T H^1_x}.
\end{equation}

Consequently, combining $(5.17)$-$(5.18)$ yields the desired estimates $(5.1)$ for $\varepsilon$ small enough. And the estimate $(5.11)$ is trivial from classical parabolic theory for $\varepsilon$ having a positive lower bound after noting that $v_\varepsilon=u_\varepsilon e^{-Nt}$.

In fact, if $k\in(1,2)$ and $k\rightarrow 1$, then we may need to choose $\varepsilon\leq \varepsilon_1$ with

$$\varepsilon_1=\frac{\exp(-|k-1|^{-2})}{C||W||_{L^\infty(\mathbb{T}^{d+1})}},$$
which implies that the uniform estimates obtained in $(5.1)$ will blow up as $k\rightarrow 1$. Moreover, for the case $k=1$, the heuristic illusion in Section 2 implies that we additionally need to assume that $\mathcal{M}_y(W)(\tau)=0$ for every $\tau\in\mathbb{T}$.

\end{proof}

\noindent \textbf{Proof of Theorem 3.3: the homogenization equation.}

Form $(5.16)$ and Lemma 5.1, we have the following uniform estimates:
\begin{equation}\begin{aligned}
||\tilde{v}_\varepsilon||_{L^2(0,T;H^1_0(\Omega_T))}+
||\partial_t \tilde{v}_\varepsilon||_{L^2(0,T;H^{-1}(\Omega))}\leq C,
\end{aligned}\end{equation}
which after using Lemma 4.3 implies that there exists $v_0\in L^2(0,T;H_0^1(\Omega))$ such that
\begin{equation}\begin{aligned}
\tilde{v}_\varepsilon\rightharpoonup v_0\text{ weakly in }L^2(0,T;H_0^1(\Omega));\\
\tilde{v}_\varepsilon\rightarrow v_0\text{ strongly in }L^2(\Omega_T).
\end{aligned}\end{equation}
In view of the expression $(5.14)$ of $\tilde{v}_\varepsilon$, we also have
\begin{equation}\begin{aligned}
{v}_\varepsilon\rightharpoonup v_0\text{ weakly in }L^2(0,T;H_0^1(\Omega));\\
{v}_\varepsilon\rightarrow v_0\text{ strongly in }L^2(\Omega_T).
\end{aligned}\end{equation}

In order to obtain the homogenization limit, we need to rewrite the term $\varepsilon^{-1}W^\varepsilon_3 v_\varepsilon$ in a more suitable form. Actually, after a direct computation, we have
\begin{equation}\begin{aligned}
\varepsilon^{-1}W^\varepsilon_3 v_\varepsilon=\varepsilon v_\varepsilon\Delta_x\chi_3^\varepsilon= \Delta_x(\varepsilon\chi_3^\varepsilon v_\varepsilon)-\varepsilon\operatorname{div}_x(\chi_3^\varepsilon\nabla v_\varepsilon)-\nabla_y\chi_3^\varepsilon\nabla v_\varepsilon.\\
\end{aligned}\end{equation}

Moreover, for every $\tau\in\mathbb{T}$, we introduce the corrector $\widehat{\chi_{3-1}}(y,\tau)$ solving the following equation
\begin{equation}\left\{\begin{aligned}
\Delta_y {\widehat{\chi_{3-1}}}=\chi_3\quad\quad\quad\text{ in }\quad\quad\quad\mathbb{T}^{d},\\
{\widehat{\chi_{3-1}}}\text{ is 1-periodic},\ \mathcal{M}_y({\widehat{\chi_{3-1}}})(\tau)=0.
\end{aligned}\right.\end{equation}
By elliptic theory, we have
\begin{equation*}
\widehat{\chi_{3-1}}\in L^\infty(\mathbb{T};W^{2,p}(\mathbb{T}^{d})),\text{ for any }1<p<\infty
\end{equation*}
and by embedding,
$$\widehat{\chi_{3-1}}\in L^\infty(\mathbb{T};C^{1+\alpha}(\mathbb{T}^{d})), \text{ for any }0<\alpha<1.$$
Moreover, according to $(5.4)$, we have
\begin{equation*}\begin{aligned}
&\partial_\tau\widehat{\chi_{3-1}}\in L^\infty(\mathbb{T};W^{2,p}(\mathbb{T}^{d})),\text{ for any }1<p<\infty;\\
&\partial_\tau\widehat{\chi_{3-1}} \in L^\infty(\mathbb{T};C^{1+\alpha}(\mathbb{T}^{d})), \text{ for any }0<\alpha<1.
\end{aligned}\end{equation*}

Then it is easy to see that
\begin{equation}\begin{aligned}
-\nabla_y\chi_3^\varepsilon\nabla v_\varepsilon=&-\varepsilon\operatorname{div}_x(\chi_3^\varepsilon\nabla v_\varepsilon)+\varepsilon\chi_3^\varepsilon\Delta v_\varepsilon\\
=&-\varepsilon\operatorname{div}_x(\chi_3^\varepsilon\nabla v_\varepsilon)+\varepsilon \partial_t(\chi_3^\varepsilon v_\varepsilon)-\varepsilon^{1-k}v_\varepsilon\partial_\tau \chi_3^\varepsilon\\
&\quad-\chi_3^\varepsilon W^\varepsilon v_\varepsilon-\varepsilon \chi_3^\varepsilon \tilde{f}+\varepsilon N\chi_3^\varepsilon v_\varepsilon,
\end{aligned}\end{equation}
where
\begin{equation}\begin{aligned}-\varepsilon^{1-k}v_\varepsilon\partial_\tau \chi_3^\varepsilon=-\varepsilon^{3-k}v_\varepsilon\Delta_x \partial_\tau \widehat{\chi_{3-1}}^\varepsilon =-\varepsilon^{2-k}\operatorname{div}_x(v_\varepsilon\nabla_y \partial_\tau \widehat{\chi_{3-1}}^\varepsilon)+\varepsilon^{2-k}\nabla v_\varepsilon\nabla_y \partial_\tau \widehat{\chi_{3-1}}^\varepsilon.\end{aligned}\end{equation}

Then, \begin{equation}\tilde{\tilde{v}}_\varepsilon=:v_\varepsilon+\sum_{i=1}^{I_k}(-1)^i\varepsilon^{i(k-1)}\chi_{3-i}^\varepsilon v_\varepsilon+\varepsilon\chi_3^\varepsilon v_\varepsilon\end{equation} satisfies the following equation:
\begin{equation}\left\{\begin{aligned}
\partial_t \tilde{\tilde{v}}_\varepsilon-\Delta \tilde{\tilde{v}}_\varepsilon+N\tilde{\tilde{v}}_\varepsilon&=f_{32}(v_\varepsilon),& &\text{in }\Omega\times (0,T),\\
\tilde{\tilde{v}}_\varepsilon&=0,& &\text{on }\partial\Omega\times (0,T),\\
\tilde{\tilde{v}}_\varepsilon&=g+\varepsilon\chi_3^\varepsilon g,&\quad &\text{on }\Omega\times \{t=0\},
\end{aligned}\right.\end{equation}
with $$\begin{aligned}f_{32}(v_\varepsilon)=&\sum_{i=1}^{I_k}(-1)^i\varepsilon^{i(k-1)}\operatorname{div}_x\left(\chi_{3-i}^\varepsilon v_\varepsilon\nabla_y\chi_3^\varepsilon\right)+\sum_{i=1}^{I_k}(-1)^{i+1}\varepsilon^{i(k-1)}\chi_{3-i}^\varepsilon\nabla_y\chi_3^\varepsilon \nabla v_\varepsilon\\
&\quad-2\varepsilon\operatorname{div}_x(\chi_3^\varepsilon\nabla v_\varepsilon)-\chi_3^\varepsilon W^\varepsilon v_\varepsilon+\tilde{f}+\sum_{i=1}^{I_k}(-1)^i\varepsilon^{i(k-1)}\chi_{3-i}^\varepsilon\tilde{f}-\varepsilon \chi_3^\varepsilon \tilde{f}\\
&\quad\quad-\varepsilon^{2-k}\operatorname{div}_x(v_\varepsilon\nabla_y \partial_\tau \widehat{\chi_{3-1}}^\varepsilon)+\varepsilon^{2-k}\nabla v_\varepsilon\nabla_y \partial_\tau \widehat{\chi_{3-1}}^\varepsilon+2\varepsilon \partial_t(\chi_3^\varepsilon v_\varepsilon)\\
&\quad\quad\quad+(-1)^{I_k+1}\varepsilon^{(I_k+1)(k-1)-1}\chi_{3-(I_k+1)}^\varepsilon W_4^\varepsilon v_\varepsilon.
\end{aligned}$$

In view of $(5.26)$ and Lemma 4.2, we know that
\begin{equation}\begin{aligned}
\tilde{\tilde{v}}_\varepsilon\rightharpoonup v_0\text{ weakly in }L^2(0,T;H_0^1(\Omega));\\
\tilde{\tilde{v}}_\varepsilon\rightarrow v_0\text{ strongly in }L^2(\Omega_T).
\end{aligned}\end{equation}
Consequently, $v_0$ satisfies the homogenized equation:
\begin{equation}\left\{\begin{aligned}
\partial_t {v}_0-\Delta {v}_0+\left(N+\mathcal{M}(\chi_3 W)\right)v_0&=\tilde{f},& &\quad\text{in }\Omega\times (0,T),\\
{v}_0&=0,& &\quad\text{on }\partial\Omega\times (0,T),\\
{v}_0&=g,& &\quad\text{on }\Omega\times \{t=0\},
\end{aligned}\right.\end{equation}
which completes the first part of Theorem 3.2.
\qed\\

Now, we are ready to bound the error.\\
\noindent \textbf{Proof of Theorem 3.3: the convergence rates}.\\

\noindent To obtain the convergence rates, we first rewrite the term $\varepsilon^{-1}W^\varepsilon_3 v_\varepsilon$ as below:
\begin{equation*}\begin{aligned}
\varepsilon^{-1}W_3^\varepsilon\cdot v_\varepsilon=\varepsilon \Delta_x\chi_3^\varepsilon v_\varepsilon= \Delta_x(\varepsilon\chi_3^\varepsilon v_\varepsilon)-\varepsilon \chi_3^\varepsilon \Delta_x v_\varepsilon-2\nabla_y\chi_3^\varepsilon\nabla v_\varepsilon\\
\end{aligned}\end{equation*}
and due to $(5.2)$, we have
\begin{equation*}\begin{aligned}
&-\varepsilon \chi_3^\varepsilon \Delta_x v_\varepsilon=-\varepsilon \chi_3^\varepsilon(\partial_t v_\varepsilon+Nv_\varepsilon-\varepsilon^{-1}W^\varepsilon v_\varepsilon-\tilde{f})\\
=&-\varepsilon \partial_t(\chi_3^\varepsilon v_\varepsilon)+\varepsilon^{1-k}v_\varepsilon\partial_\tau \chi_3^\varepsilon+\chi_3^\varepsilon W^\varepsilon v_\varepsilon+\varepsilon \chi_3^\varepsilon \tilde{f}-\varepsilon N\chi_3^\varepsilon v_\varepsilon\\
=&-\varepsilon \partial_t(\chi_3^\varepsilon v_\varepsilon)+\chi_3^\varepsilon W^\varepsilon v_\varepsilon+\varepsilon \chi_3^\varepsilon \tilde{f}-\varepsilon N\chi_3^\varepsilon v_\varepsilon\\
&\quad\quad+\varepsilon^{2-k}\operatorname{div}_x(v_\varepsilon\nabla_y \partial_\tau \widehat{\chi_{3-1}}^\varepsilon)-\varepsilon^{2-k}\nabla v_\varepsilon\nabla_y \partial_\tau \widehat{\chi_{3-1}}^\varepsilon.
\end{aligned}\end{equation*}

Then with $\tilde{\tilde{v}}_\varepsilon$ defined in $(5.26)$, $$w_\varepsilon=\tilde{\tilde{v}}_\varepsilon-v_0$$  satisfies
\begin{equation}\left\{\begin{aligned}
\partial_t w_\varepsilon-\Delta w_\varepsilon+Nw_\varepsilon&=f_{33}(v_\varepsilon),& \quad&\text{in }\Omega\times (0,T),\\
w_\varepsilon&=0,& &\text{on }\partial\Omega\times (0,T),\\
w_\varepsilon&=\varepsilon\chi_3^\varepsilon g,& &\text{on }\Omega\times \{t=0\},
\end{aligned}\right.\end{equation}
with \begin{equation}\begin{aligned}f_{33}(v_\varepsilon)=&\sum_{i=1}^{I_k}(-1)^i\varepsilon^{i(k-1)}\chi_{3-i}^\varepsilon\tilde{f}
+\varepsilon \chi_3^\varepsilon \tilde{f}-2\nabla_y\chi_3^\varepsilon\nabla v_\varepsilon+\chi_3^\varepsilon W_\varepsilon^\varepsilon v_\varepsilon\\
&+\varepsilon^{2-k}\operatorname{div}_x(v_\varepsilon\nabla_y \partial_\tau \widehat{\chi_{3-1}}^\varepsilon)-\varepsilon^{2-k}\nabla v_\varepsilon\nabla_y \partial_\tau \widehat{\chi_{3-1}}^\varepsilon+\mathcal{M}(\chi_3 W)v_0\\
&\quad\sum_{i=1}^{I_k}(-1)^i\varepsilon^{i(k-1)}\operatorname{div}_x\left(\chi_{3-i}^\varepsilon v_\varepsilon\nabla_y\chi_3^\varepsilon\right)+\sum_{i=1}^{I_k}(-1)^{i+1}\varepsilon^{i(k-1)}\chi_{3-i}^\varepsilon\nabla_y\chi_3^\varepsilon \nabla v_\varepsilon\\
&\quad\quad\quad+(-1)^{I_k+1}\varepsilon^{(I_k+1)(k-1)-1}\chi_{3-(I_k+1)}^\varepsilon W_4^\varepsilon v_\varepsilon.
\end{aligned}\end{equation}

To proceed, rewrite $w_\varepsilon=w_{3,\varepsilon}+w_{4,\varepsilon}$, with $w_{4,\varepsilon}$ defied as
\begin{equation}\left\{\begin{aligned}
\partial_t w_{4,\varepsilon}-\Delta w_{4,\varepsilon}+Nw_{4,\varepsilon}&=f_{34}(v_\varepsilon),&\quad &\text{in }\Omega\times (0,T),\\
w_{4,\varepsilon}&=0,& &\text{on }\partial\Omega\times (0,T),\\
w_{4,\varepsilon}&=0,& &\text{on }\Omega\times \{t=0\},
\end{aligned}\right.\end{equation}
with \begin{equation*}\begin{aligned}f_{34}(v_\varepsilon)=&\sum_{i=1}^{I_k}(-1)^i\varepsilon^{i(k-1)}\chi_{3-i}^\varepsilon\tilde{f}
+\varepsilon \chi_3^\varepsilon \tilde{f}
+\varepsilon^{2-k}\operatorname{div}_x(v_\varepsilon\nabla_y \partial_\tau \widehat{\chi_{3-1}}^\varepsilon)\\
&-\varepsilon^{2-k}\nabla v_\varepsilon\nabla_y \partial_\tau \widehat{\chi_{3-1}}^\varepsilon+
\sum_{i=1}^{I_k}(-1)^i\varepsilon^{i(k-1)}\operatorname{div}_x\left(\chi_{3-i}^\varepsilon v_\varepsilon\nabla_y\chi_3^\varepsilon\right)\\
&+\sum_{i=1}^{I_k}(-1)^{i+1}\varepsilon^{i(k-1)}\chi_{3-i}^\varepsilon\nabla_y\chi_3^\varepsilon \nabla v_\varepsilon+(-1)^{I_k+1}\varepsilon^{(I_k+1)(k-1)-1}\chi_{3-(I_k+1)}^\varepsilon W_4^\varepsilon v_\varepsilon.
\end{aligned}\end{equation*}
Then, we have
\begin{equation}\begin{aligned}
&||\partial_t w_{3,\varepsilon}||_{L^2(\Omega_T)}\leq C\left(||f||_{L^2(\Omega_T)}+||g||_{H^1(\Omega)}\right);\\
&||\partial_t w_{4,\varepsilon}||_{L^2(0,T;H^{-1}(\Omega))}\leq C\max(\varepsilon^{2-k},\varepsilon^{k-1})\left(||f||_{L^2(\Omega_T)}+||g||_{H^1(\Omega)}\right),\end{aligned}\end{equation}
Moreover, we have the following estimates for $w_{4,\varepsilon}$:
\begin{equation}
||w_{4,\varepsilon}||_{L^\infty_TL^2_x}+||\nabla w_{4,\varepsilon}||_{L^2(\Omega_T)}+|| w_{4,\varepsilon}||_{L^2(\Omega_T)}\leq C\max(\varepsilon^{2-k},\varepsilon^{k-1})\left(||f||_{L^2(\Omega_T)}+||g||_{H^1(\Omega)}\right).
\end{equation}

To estimate $w_{3,\varepsilon}$, we note that $w_{3,\varepsilon}$ satisfies
\begin{equation}\left\{\begin{aligned}
\partial_t w_{3,\varepsilon}-\Delta w_{3,\varepsilon}+Nw_{3,\varepsilon}&=f_{35}(v_\varepsilon),&\quad &\text{in }\Omega\times (0,T),\\
w_{3,\varepsilon}&=0,& &\text{on }\partial\Omega\times (0,T),\\
w_{3,\varepsilon}&=\varepsilon\chi_3^\varepsilon g,& &\text{on }\Omega\times \{t=0\},
\end{aligned}\right.\end{equation}
with
 \begin{equation*}\begin{aligned}f_{35}(v_\varepsilon)=-2\nabla_y\chi_3^\varepsilon\nabla v_\varepsilon+\chi_3^\varepsilon W^\varepsilon v_\varepsilon+\mathcal{M}(\chi_3 W)v_0.
\end{aligned}\end{equation*}
Firstly, a direct computation shows that
\begin{equation}\begin{aligned}
-\int_{\Omega_T}2\nabla_y\chi_3^\varepsilon\nabla v_\varepsilon\cdot w_{3,\varepsilon}=&-\int_{\Omega_T}2\nabla_y\chi_3^\varepsilon\nabla\left(w_{4,\varepsilon} -\sum_{i=1}^{I_k}(-1)^i\varepsilon^{i(k-1)}\chi_{3-i}^\varepsilon v_\varepsilon\right)\cdot w_{3,\varepsilon}\\
&+2\int_{\Omega_T}\nabla_y\chi_3^\varepsilon\nabla\left(\varepsilon\chi_3^\varepsilon v_\varepsilon-v_0\right)\cdot w_{3,\varepsilon}-\int_{\Omega_T}2\nabla_y\chi_3^\varepsilon\nabla w_{3,\varepsilon}\cdot w_{3,\varepsilon}\\
=&:E_6+E_7+E_8.
\end{aligned}\end{equation}
According to $(5.34)$ after noting that $\chi_{3-i}$ being independent of $y$, it is easy to see that
\begin{equation}
|E_6|\leq C\max(\varepsilon^{2-k},\varepsilon^{k-1})\left(||f||_{L^2(\Omega_T)}+||g||_{H^1(\Omega)}\right)||w_{3,\varepsilon}||_{L^2(\Omega_T)},
\end{equation}and

\begin{equation}
|E_8|\leq \frac 1{16}||\nabla w_{3,\varepsilon}||_{L^2(\Omega_T)}^2+C||w_{3,\varepsilon}||_{L^2(\Omega_T)}^2,
\end{equation}
and
\begin{equation}\begin{aligned}
&E_{72}=:-2\int_{\Omega_T}\nabla_y\chi_3^\varepsilon\nabla v_0\cdot w_{3,\varepsilon}\\
\leq &C\varepsilon||\Delta v_0||_{L^2(\Omega_T)}||w_{3,\varepsilon}||_{L^2(\Omega_T)}
+C\varepsilon||\nabla v_0||_{L^2(\Omega_T)}||\nabla w_{3,\varepsilon}||_{L^2(\Omega_T)}.
\end{aligned}\end{equation}
For the estimate of $E_{71}$, a direct computation yields that
\begin{equation}\begin{aligned}
E_{71}=:2\int_{\Omega_T}\nabla_y\chi_3^\varepsilon\nabla\left(\varepsilon\chi_3^\varepsilon v_\varepsilon\right)\cdot w_{3,\varepsilon}=2\int_{\Omega_T}|\nabla_y\chi_3^\varepsilon|^2 v_\varepsilon\cdot w_{3,\varepsilon}+2\int_{\Omega_T}\varepsilon\chi_3^\varepsilon \nabla_y\chi_3^\varepsilon\nabla v_\varepsilon\cdot w_{3,\varepsilon}.
\end{aligned}\end{equation}

To proceed, we note that $\mathcal{M}(|\nabla_y\chi_3|^2)=-\mathcal{M}(\chi_3 W)=-\mathcal{M}(\chi_3 W_4)$. The first equality is due to multiplying the equation $(5.3)$ by $\chi_3$ and integrating over $\mathbb{T}^{d+1}$. And the second equality is due to $\mathcal{M}(\chi_3 W_4)=\mathcal{M}_\tau(\mathcal{M}_y(\chi_3) W_4)=0$. Then, we can introduce the following cell problem:
\begin{equation}\left\{\begin{aligned}
\partial_\tau\widehat{\chi_{3-2}} -\Delta_y\widehat{\chi_{3-2}}=2|\nabla_y\chi_3|^2+\mathcal{M}(\chi_3 W)+\chi_3 W\text{ in }\mathbb{T}^{d+1},\\
\widehat{\chi_{3-2}}\text{ is 1-periodic in }(y,\tau)\quad\text{ with }\quad\mathcal{M}(\widehat{\chi_{3-2}})=0.
\end{aligned}\right.\end{equation}
Then $$\widehat{\chi_{3-2}}\in W^{2,1}_p(\mathbb{T}^{d+1}),\text{ for any }1<p<\infty,$$
and by embedding,
$$\widehat{\chi_{3-2}} \in C^{1+\alpha,\frac{1+\alpha}2}(\mathbb{T}^{d+1}), \text{ for any }0<\alpha<1.$$

\noindent
In view of $(5.41)$,
a simple computation shows that
\begin{equation}\begin{aligned}
&\int_{\Omega_T}\left(2|\nabla_y\chi_3^\varepsilon|^2v_\varepsilon+\mathcal{M}(\chi_3 W)v_0+\chi_3^\varepsilon W^\varepsilon v_\varepsilon\right)w_{3,\varepsilon}\\
=&\int_{\Omega_T}\left(2|\nabla_y\chi_3^\varepsilon|^2+\mathcal{M}(\chi_3 W)+\chi_3^\varepsilon W^\varepsilon \right)v_0w_{3,\varepsilon}\\
&\quad\quad+\int_{\Omega_T}\left(2|\nabla_y\chi_3^\varepsilon|^2+\chi_3^\varepsilon W^\varepsilon\right)\left( v_\varepsilon-v_0\right)w_{3,\varepsilon}\\
=&:\int_{\Omega_T} \left(\varepsilon^k\partial_t\widehat{\chi_{3-2}}^\varepsilon-\varepsilon^2\Delta_x\widehat{\chi_{3-2}}^\varepsilon\right)
v_0w_{3,\varepsilon}
+E_9=:E_{10}+E_9.
\end{aligned}\end{equation}
It is then easy to see that
\begin{equation}\begin{aligned}
|E_9|\leq& C||w_{3,\varepsilon}||_{L^2(\Omega_T)}^2
+C\varepsilon^{k-1}||v_\varepsilon||_{L^2(\Omega_T)}||w_{3,\varepsilon}||_{L^2(\Omega_T)}
+C||w_{3,\varepsilon}||_{L^2(\Omega_T)}||w_{4,\varepsilon}||_{L^2(\Omega_T)}\\
\leq& C||w_{3,\varepsilon}||_{L^2(\Omega_T)}^2+C\varepsilon^{2k-2}||v_\varepsilon||_{L^2(\Omega_T)}^2
+C||w_{4,\varepsilon}||_{L^2(\Omega_T)}^2.
\end{aligned}\end{equation}
And a direct computation yields that
\begin{equation}\begin{aligned}
|E_{10}|\leq& C\varepsilon^k||v_0||_{L^\infty_TL^2_x}||w_{3,\varepsilon}||_{L^\infty_TL^2_x}
+C\varepsilon^k||\partial_tv_0||_{L^2(\Omega_T)}||w_{3,\varepsilon}||_{L^2(\Omega_T)}\\
&+C\varepsilon^k||v_0||_{L^2(\Omega_T)}||\partial_t w_{3,\varepsilon}||_{L^2(\Omega_T)}+C\varepsilon|| v_0||_{L^2(0,T;H^1(\Omega))}|| w_{3,\varepsilon}||_{L^2(0,T;H^1(\Omega))}\\
\leq& \frac1{16}||w_{3,\varepsilon}||_{L^\infty_TL^2_x}^2+\frac1{16}||\nabla w_{3,\varepsilon}||_{L^2(\Omega_T)}^2+|| w_{3,\varepsilon}||_{L^2(\Omega_T)}^2\\
&+C\varepsilon^{k}\left(||f||_{L^2(\Omega_T)}^2+||g||_{H^1(\Omega)}^2\right),
\end{aligned}\end{equation}
where we have used $(5.33)_1$ in the inequality above.

Therefore, multiplying the equation $(5.35)$ by $w_{3,\varepsilon}$ after choosing $N$ suitable large yields that
\begin{equation}
||w_{3,\varepsilon}||_{L^\infty_TL^2_x}+||\nabla w_{3,\varepsilon}||_{L^2(\Omega_T)}+|| w_{3,\varepsilon}||_{L^2(\Omega_T)}\leq C\max(\varepsilon^{k-1},\varepsilon^{2-k})\left(||f||_{L^2(\Omega_T)}+||g||_{H^1(\Omega)}\right),
\end{equation}
which yields the desired estimates $(3.9)$ after combining $(5.34)$.\qed

\subsection{Proof of Theorem 3.4}
First, we always assume that $0<  k\leq 1$, $\gamma=1$, $\partial_\tau W\in \mathbb{T}^{d+1}$ and $\mathcal{M}_y(W)(\tau)=0$ for every $\tau\in\mathbb{T}$. In fact, under these conditions, we know that $W_4$ defined in $(5.2)$ equals to 0, and $W_3=W$. And then $\chi_{3-i}=0$ for $i=1,2,\cdots, I_k$.

Therefore, due to $(5.3)$, we know that
\begin{equation*}\begin{aligned}
\varepsilon^{-1}W^\varepsilon v_\varepsilon=\varepsilon \Delta_x\chi_3^\varepsilon\cdot v_\varepsilon= \operatorname{div}_x(v_\varepsilon\nabla_y\chi_3^\varepsilon)-\nabla_y\chi_3^\varepsilon\nabla v_\varepsilon.
\end{aligned}\end{equation*}
The uniform estimates follow directly from the equality above and the effective equation is given by
\begin{equation*}\left\{\begin{aligned}
\partial_t {v}_0-\Delta {v}_0+\left(N+\mathcal{M}(\chi_3 W)\right)v_0&=\tilde{f},&\quad &\text{in }\Omega\times (0,T),\\
{v}_0&=0,& &\text{on }\partial\Omega\times (0,T),\\
{v}_0&=g,& &\text{on }\Omega\times \{t=0\}.
\end{aligned}\right.\end{equation*}
By setting
\begin{equation}w_\varepsilon=v_\varepsilon+\varepsilon\chi_3^\varepsilon v_\varepsilon-v_0,\end{equation}
and due to $(5.30)$ after noting that $\chi_{3-i}=0$ for $i\in \mathbb{N}_+$,
$w_\varepsilon$ satisfies

\begin{equation}\left\{\begin{aligned}
\partial_t w_\varepsilon-\Delta w_\varepsilon+Nw_\varepsilon&=f_{36}(v_\varepsilon),&\quad &\text{in }\Omega\times (0,T),\\
w_\varepsilon&=0,& &\text{on }\partial\Omega\times (0,T),\\
w_\varepsilon&=\varepsilon\chi_3^\varepsilon g,& &\text{on }\Omega\times \{t=0\},
\end{aligned}\right.\end{equation}
with
\begin{equation}\begin{aligned}f_{36}(v_\varepsilon)=&\varepsilon \chi_3^\varepsilon \tilde{f}-2\nabla_y\chi_3^\varepsilon\nabla v_\varepsilon+\chi_3^\varepsilon W_\varepsilon^\varepsilon v_\varepsilon+\mathcal{M}(\chi_3 W)v_0\\
&+\varepsilon^{2-k}\operatorname{div}_x(v_\varepsilon\nabla_y \partial_\tau \widehat{\chi_{3-1}}^\varepsilon)
-\varepsilon^{2-k}\nabla v_\varepsilon\nabla_y \partial_\tau \widehat{\chi_{3-1}}^\varepsilon.\\
\end{aligned}\end{equation}

In order to obtain the convergence, similar to the proof of $(5.33)$ and  $(5.42)$, we need only to the bound the error for the following equation:
\begin{equation}\left\{\begin{aligned}
\partial_t \rho_\varepsilon-\Delta \rho_\varepsilon+N\rho_\varepsilon&=f_{37}(v_0),&\quad &\text{in }\Omega\times (0,T),\\
\rho_\varepsilon&=0,& &\text{on }\partial\Omega\times (0,T),\\
\rho_\varepsilon&=\varepsilon\chi_3^\varepsilon g,& &\text{on }\Omega\times \{t=0\},
\end{aligned}\right.\end{equation}
with
\begin{equation}\begin{aligned}f_{37}(v_0)=&\left(2|\nabla_y\chi_3^\varepsilon|^2+\mathcal{M}(\chi_3 W)+\chi_3^\varepsilon W^\varepsilon \right)v_0\\
=&\left(2|\nabla_y\chi_3^\varepsilon|^2+\mathcal{M}_y(\chi_3 W)(t/\varepsilon^k)+\chi_3^\varepsilon W^\varepsilon \right)v_0+\left(\mathcal{M}(\chi_3 W)-\mathcal{M}_y(\chi_3 W)(t/\varepsilon^k)\right)v_0\\
=&:f_{371}(v_0)+f_{372}(v_0).
\end{aligned}\end{equation}
According to $(5.3)$, we know that
\begin{equation}\begin{aligned}
&\mathcal{M}_y\left(2|\nabla_y\chi_3|^2+\mathcal{M}_y(\chi_3 W)+\chi_3 W \right)(\tau)=0,\text{ for every }\tau\in \mathbb{T};\\
&\mathcal{M}_\tau\left(\mathcal{M}(\chi_3 W)-\mathcal{M}_y(\chi_3 W)\right)=0.
\end{aligned}\end{equation}
To proceed, for every $\tau\in\mathbb{T}$, we introduce $\widehat{\chi_{3-3}}(y,\tau)$ satisfying

\begin{equation}\left\{\begin{aligned}
\Delta_y \widehat{\chi_{3-3}}=\left(2|\nabla_y\chi_3|^2+\mathcal{M}_y(\chi_3 W)+\chi_3 W \right)\text{ in }\mathbb{T}^{d},\\
{\widehat{\chi_{3-3}}}\text{ is 1-periodic in }(y,\tau),\ \mathcal{M}_y({\widehat{\chi_{3-3}}})(\tau)=0.
\end{aligned}\right.\end{equation}
By elliptic theory, we have
\begin{equation*}
\widehat{\chi_{3-3}}\in L^\infty(\mathbb{T};W^{2,p}(\mathbb{T}^{d})),\text{ for any }1<p<\infty
\end{equation*}
and by embedding,
$$\widehat{\chi_{3-3}}\in L^\infty(\mathbb{T};C^{1+\alpha}(\mathbb{T}^{d})), \text{ for any }0<\alpha<1.$$
And we introduce $\widehat{\chi_{3-3}}(\tau)\in L^\infty(\mathbb{T})$ satisfying
\begin{equation}
\widehat{\chi_{3-4}}(\tau)=\int_0^\tau \left(\mathcal{M}(\chi_3 W)-\mathcal{M}_y(\chi_3 W)\right)(s)ds.
\end{equation}
Then, it is easy to see that
\begin{equation}\begin{aligned}
f_{371}(v_0)=\varepsilon^2 v_0\Delta_x \widehat{\chi_{3-3}}^\varepsilon=
\varepsilon\operatorname{div}\left(v_0\nabla_y \widehat{\chi_{3-3}}^\varepsilon \right)
-\varepsilon\nabla v_0\nabla_y \widehat{\chi_{3-3}}^\varepsilon,
\end{aligned}\end{equation}
and
\begin{equation}\begin{aligned}
f_{372}(v_0)=\varepsilon^k v_0\partial_t \widehat{\chi_{3-4}}^\varepsilon=
\varepsilon^k\partial_t\left( v_0 \widehat{\chi_{3-4}}^\varepsilon\right)-\varepsilon^k\widehat{\chi_{3-4}}^\varepsilon \partial_t v_0.
\end{aligned}\end{equation}
Therefore, according to $(5.49)$-$(5.55)$, $\tilde{\rho}_\varepsilon=:\rho_\varepsilon-\varepsilon^k \widehat{\chi_{3-3}}^\varepsilon v_0$ satisfies
\begin{equation}\left\{\begin{aligned}
\partial_t \tilde{\rho}_\varepsilon-\Delta \tilde{\rho}_\varepsilon+N\tilde{\rho}_\varepsilon&=f_{38}(v_0),&\quad &\text{in }\Omega\times (0,T),\\
\tilde{\rho}_\varepsilon&=0,& &\text{on }\partial\Omega\times (0,T),\\
\tilde{\rho}_\varepsilon&=\varepsilon\chi_3^\varepsilon g,& &\text{on }\Omega\times \{t=0\},
\end{aligned}\right.\end{equation}
with
\begin{equation}\begin{aligned}f_{38}(v_0)=&-\varepsilon^k\widehat{\chi_{3-4}}^\varepsilon \partial_t v_0+
\varepsilon^k\widehat{\chi_{3-4}}^\varepsilon \Delta_x v_0-N\varepsilon^k\widehat{\chi_{3-4}}^\varepsilon v_0\\
&+\varepsilon\operatorname{div}\left(v_0\nabla_y \widehat{\chi_{3-3}}^\varepsilon \right)
-\varepsilon\nabla v_0\nabla_y \widehat{\chi_{3-3}}^\varepsilon.
\end{aligned}\end{equation}
Then, it is easy to see that
\begin{equation}
||\tilde{\rho}_\varepsilon||_{L^\infty_TL^2_x}+||\nabla\tilde{\rho}_\varepsilon||_{L^2(\Omega_T)}+|| \tilde{\rho}_\varepsilon||_{L^2(\Omega_T)}\leq C\varepsilon^{k}\left(||f||_{L^2(\Omega_T)}+||g||_{H^1(\Omega)}\right).
\end{equation}

Consequently, with $w_\varepsilon$ defined in $(5.46)$, we can obtain the following desired convergence rates:
\begin{equation}
||w_{\varepsilon}||_{L^\infty_TL^2_x}+||\nabla w_{\varepsilon}||_{L^2(\Omega_T)}+|| w_{\varepsilon}||_{L^2(\Omega_T)}\leq C\varepsilon^k\left(||f||_{L^2(\Omega_T)}+||g||_{H^1(\Omega)}\right).
\end{equation}

Moreover, if $k=0$, according to $(5.27)$, then it is easy to see that the homogenized equation for $(1.1)$ is given by
\begin{equation}\left\{\begin{aligned}
\partial_t {v}_0-\Delta {v}_0+\left(N+\mathcal{M}_y(\chi_3 W)(t)\right)v_0&=\tilde{f},&\quad &\text{in }\Omega\times (0,T),\\
{v}_0&=0,& &\text{on }\partial\Omega\times (0,T),\\
{v}_0&=g,& &\text{on }\Omega\times \{t=0\}.
\end{aligned}\right.\end{equation}

Note that if $k=0$, then do not need to handle the second term in $(5.50)$.
 Similar to the proof of $(5.54)$, we can arrive at the desired estimates:
\begin{equation}
||w_{\varepsilon}||_{L^\infty_TL^2_x}+||\nabla w_{\varepsilon}||_{L^2(\Omega_T)}+|| w_{\varepsilon}||_{L^2(\Omega_T)}\leq C\varepsilon\left(||f||_{L^2(\Omega_T)}+||g||_{H^1(\Omega)}\right),
\end{equation}
by letting $k=0$ in $(5.46)$.

\section{Proof of Theorem 3.6}

In this subsection, we always assume that $2<k\leq 3$, $\gamma=k-1$, $W\in L^\infty(\mathbb{T}^{d+1})$, $\nabla_yW,\nabla ^2_yW,\nabla ^3_yW\in L^\infty(\mathbb{T}^{d+1})$ with $\mathcal{M}_\tau(W)(y)=0$ for every $y\in \mathbb{T}^d$. And we first derive a uniform estimate for $u_\varepsilon$.
\begin{lemma}[Uniform estimates]
Under the conditions in Theorem 3.6, there hold the following uniform estimates for $u_\varepsilon$:
\begin{equation}
||u_\varepsilon||_{L^\infty_TL^2_x}+||\nabla u_\varepsilon||_{L^2(\Omega_T)}+|| u_\varepsilon||_{L^2(\Omega_T)}\leq C\left(||g||_{L^2(\Omega)}+||f||_{L^2(\Omega)}\right),
\end{equation}
for a positive constant $C$ depending only on $W$, $d$ and $T$.
\end{lemma}
\begin{proof}Recall that $v_\varepsilon$ satisfies
\begin{equation}\left\{\begin{aligned}
\partial_t v_\varepsilon-\Delta v_\varepsilon+\left(N-\varepsilon^{1-k}W^\varepsilon\right)v_\varepsilon&=\tilde{f}, \quad\quad\text{in }\Omega\times (0,T),\\
v_\varepsilon&=0, \quad\quad\text{on }\partial\Omega\times (0,T),\\
v_\varepsilon&=g, \quad\quad\text{on }\Omega\times \{t=0\}.
\end{aligned}\right.\end{equation}
Multiplying the equation $(6.2)$ by $v_\varepsilon$ and integrating over $\Omega\times (0,t)$ for $t\in(0,T]$ yield that
\begin{equation}
\frac12\int_{\Omega\times \{t\}} |v_\varepsilon|^2+\int_0^t\int_\Omega |\nabla v_\varepsilon|^2+\int_0^t\int_\Omega \left(N-\varepsilon^{1-k}W^\varepsilon\right) |v_\varepsilon|^2=\int_0^t\int_\Omega \tilde{f}v_\varepsilon+\frac12\int_\Omega |g|^2.\end{equation}
With $\chi_5$ defined in $(3.17)$, a direct computation yields that
\begin{equation}\begin{aligned}
&-\varepsilon^{1-k}\int_0^t\int_\Omega W^\varepsilon |v_\varepsilon|^2
=-\varepsilon\int_0^t\int_\Omega|v_\varepsilon|^2\cdot\partial_t \chi_5^\varepsilon\\
=&-\varepsilon\int_{\Omega\times \{t\}}|v_\varepsilon|^2\cdot\chi_5^\varepsilon+2\varepsilon
\int_0^t\int_\Omega v_\varepsilon\cdot\partial_tv_\varepsilon\cdot \chi_5^\varepsilon=:M_1+M_2.
\end{aligned}\end{equation}
In view of $(6.2)$, we have
$$M_2=2\varepsilon
\int_0^t\int_\Omega v_\varepsilon\chi_5^\varepsilon\cdot\left(\Delta v_\varepsilon+\varepsilon^{1-k}W^\varepsilon v_\varepsilon+\tilde{f}-N v_\varepsilon\right)=:M_{21}+M_{22}+M_{23}.$$
It is easy to see that
\begin{equation}|M_{21}|\leq 2 \varepsilon||\chi_5||_{L^\infty}\int_0^t\int_\Omega |\nabla v_\varepsilon|^2+
2 ||\nabla_y\chi_5||_{L^\infty}\int_0^t\int_\Omega |\nabla v_\varepsilon|\cdot|v_\varepsilon|,\end{equation}
and
\begin{equation}|M_{23}|\leq 2 \varepsilon||\chi_5||_{L^\infty}\int_0^t\int_\Omega | v_\varepsilon|\cdot|\tilde{f}|+2 \varepsilon N||\chi_5||_{L^\infty}\int_0^t\int_\Omega | v_\varepsilon|^2.\end{equation}
To estimate $M_{22}$, we let
$$\chi_{6-1}(y,t)=:\int_0^t (W\chi_5)(y,s)ds,$$
then due to $\mathcal{M}_\tau(W\chi_5)=\frac12\left(\mathcal{M}_\tau(W)\right)^2=0$ for every $y\in \mathbb{T}^d$ and $\nabla_y (W\chi_5)\in L^\infty(\mathbb{T}^{d+1})$, we have
\begin{equation}
\chi_{6-1}\in L^\infty(\mathbb{T}^{d+1})\ \text{ and }\nabla_y\chi_{6-1}\in L^\infty(\mathbb{T}^{d+1}).
\end{equation}
Therefore,
\begin{equation*}\begin{aligned}
&M_{22}=2\varepsilon^{2}\int_0^t\int_\Omega|v_\varepsilon|^2\cdot\partial_t \chi_{6-1}^\varepsilon\\
=&2\varepsilon^{2}\int_{\Omega\times \{t\}}|v_\varepsilon|^2\cdot\chi_{6-1}^\varepsilon-4\varepsilon^{2}
\int_0^t\int_\Omega v_\varepsilon\cdot\partial_tv_\varepsilon\cdot \chi_{6-1}^\varepsilon\\
=&2\varepsilon^{2}\int_{\Omega\times \{t\}}|v_\varepsilon|^2\cdot\chi_{6-1}^\varepsilon-4\varepsilon^{2}
\int_0^t\int_\Omega v_\varepsilon\chi_{6-1}^\varepsilon\cdot\left(\Delta v_\varepsilon+\varepsilon^{1-k}W^\varepsilon v_\varepsilon+\tilde{f}-Nv_\varepsilon\right)\\
=&:2\varepsilon^{2}\int_{\Omega\times \{t\}}|v_\varepsilon|^2\cdot\chi_{6-1}^\varepsilon+M_{221}+M_{222}+M_{223}.
\end{aligned}\end{equation*}
It is easy to see that
\begin{equation}\begin{aligned}
|M_{221}|\leq 4\varepsilon^{2}||\chi_{6-1}||_{L^\infty}\int_0^t\int_\Omega |\nabla v_\varepsilon|^2+4\varepsilon||\nabla_y\chi_{6-1}||_{L^\infty}\int_0^t\int_\Omega |\nabla v_\varepsilon|\cdot|v_\varepsilon|,
\end{aligned}\end{equation}
and
\begin{equation}
|M_{222}|\leq 4\varepsilon^{3-k}||W\chi_{6-1}||_{L^\infty}\int_0^t\int_\Omega | v_\varepsilon|^2,
\end{equation}
and
\begin{equation}
|M_{223}|\leq 4\varepsilon^{2}||\chi_{6-1}||_{L^\infty}\int_0^t\int_\Omega |v_\varepsilon|\cdot|\tilde{f}|+4\varepsilon^{2}N||\chi_{6-1}||_{L^\infty}\int_0^t\int_\Omega |v_\varepsilon|^2.
\end{equation}
In conclusion, combining $(6.4)$-$(6.10)$ gives that
\begin{equation}\begin{aligned}
\varepsilon^{1-k}\left|\int_0^t\int_\Omega W^\varepsilon |v_\varepsilon|^2\right|
\leq C_0\varepsilon&\int_{\Omega\times\{T\}}|v_\varepsilon|^2+C_0 \varepsilon\int_0^t\int_\Omega |\nabla v_\varepsilon|^2+\left(C_0N\varepsilon+C_0\right)\int_0^t\int_\Omega | v_\varepsilon|^2\\
&+C_0\int_0^t\int_\Omega |\nabla v_\varepsilon|\cdot|v_\varepsilon|+C_0\varepsilon N\int_0^t\int_\Omega | v_\varepsilon|\cdot|\tilde{f}|,
\end{aligned}\end{equation}
where $C_0=:4\left(||\chi_5||_{L^\infty}+||\chi_{6-1}||_{L^\infty}+||\nabla_y\chi_5||_{L^\infty}
+||\nabla_y\chi_{6-1}||_{L^\infty}+||W\chi_{6-1}||_{L^\infty}\right).$

In conclusion, substituting $(6.11)$ into $(6.3)$ after choosing $0<\varepsilon\leq \varepsilon_2$ with $C_0\varepsilon_2=\frac1{16}$ and $N>2C_0$ yields that
\begin{equation}
\int_{\Omega\times \{t\}} |v_\varepsilon|^2+\int_0^t\int_\Omega |\nabla v_\varepsilon|^2+N\int_0^t\int_\Omega |v_\varepsilon|^2\leq C\int_0^t\int_\Omega |\tilde{f}|^2+C\int_\Omega |g|^2,\end{equation}

Consequently, noting that $v_\varepsilon=e^{-2C_1t}u_\varepsilon$, we obtain, for $0<\varepsilon\leq \varepsilon_2$
\begin{equation}
\int_{\Omega\times \{t\}} |u_\varepsilon|^2+\int_0^t\int_\Omega |\nabla u_\varepsilon|^2+\int_0^t\int_\Omega |u_\varepsilon|^2\leq C\int_0^t\int_\Omega |\tilde{f}|^2+C\int_\Omega |g|^2,\end{equation}
for a positive constant $C$ depending only on $C_0$ and $T$. Thus we complete the proof of $(6.1)$ since it directly follows from Lemma 4.1 for $1\geq\varepsilon\geq \varepsilon_2$.
\end{proof}
In order to determine the effective equation, we also need the following Hessian estimates for $u_\varepsilon$.

\begin{lemma}Under the conditions in Theorem 3.6, there hold the following estimates for $u_\varepsilon$:
\begin{equation}||\nabla {u}_\varepsilon||_{L^\infty_TL^2_x}+||\nabla^2 {u}_\varepsilon||_{L^2(\Omega_T)}\leq C\varepsilon^{-1}\left(||f||_{L^2(\Omega_T)}+||g||_{H^1(\Omega)}\right),\end{equation}
for a positive constant $C$ depending only on $W$, $T$, $d$ and $\Omega$.
\end{lemma}
\begin{proof}
A direct computation yields that
$$\begin{aligned}
\varepsilon^{1-k}W^\varepsilon u_\varepsilon&=\varepsilon \partial_t \chi_5^\varepsilon \cdot u_\varepsilon=\varepsilon\partial_t (\chi_5^\varepsilon u_\varepsilon)-\varepsilon\chi_5^\varepsilon \partial_t u_\varepsilon\\
&=\varepsilon\partial_t (\chi_5^\varepsilon u_\varepsilon)-\varepsilon\chi_5^\varepsilon\left(\Delta u_\varepsilon+\varepsilon^{1-k}W^\varepsilon u_\varepsilon+f\right)\\
&=\varepsilon\partial_t (\chi_5^\varepsilon u_\varepsilon)-\varepsilon\Delta(\chi_5^\varepsilon u_\varepsilon)+2\nabla_y \chi_5^\varepsilon \nabla u_\varepsilon\\
&\quad+{\varepsilon}^{-1}\Delta_y\chi_5^\varepsilon\cdot u_\varepsilon-\varepsilon\chi_5^\varepsilon f-\varepsilon^{2-k}\chi_5^\varepsilon W^\varepsilon u_\varepsilon.
\end{aligned}$$
Then
it is easy to see that $\tilde{u}_\varepsilon=:u_\varepsilon-\varepsilon\chi^\varepsilon_5u_\varepsilon$ satisfies the following parabolic equations:
\begin{equation*}\left\{\begin{aligned}
\partial_t \tilde{u}_\varepsilon-\Delta \tilde{u}_\varepsilon &=F(u_\varepsilon),&\quad &\text{in }\Omega\times (0,T),\\
u_\varepsilon&=0,& &\text{on }\partial\Omega\times (0,T),\\
u_\varepsilon&=g,& &\text{on }\Omega\times \{t=0\},
\end{aligned}\right.\end{equation*}
with $$F(u_\varepsilon)=:(1-\varepsilon\chi^\varepsilon_6)f-\varepsilon^{2-k}\chi_5^\varepsilon W^\varepsilon u_\varepsilon+\varepsilon^{-1}u_\varepsilon\Delta_y\chi_5^\varepsilon+2\nabla u_\varepsilon\nabla_y\chi^\varepsilon_5.$$
Then
according to Lemma 4.1, we have
\begin{equation}
||\partial_t \tilde{u}_\varepsilon||_{L^2(\Omega_T)}+||\nabla \tilde{u}_\varepsilon||_{L^\infty_TL^2_x}+||\nabla^2 \tilde{u}_\varepsilon||_{L^2(\Omega_T)}\leq C\varepsilon^{-1}\left(||f||_{L^2(\Omega_T)}+||g||_{H^1(\Omega)}\right).
\end{equation}
Therefore, for $\varepsilon\leq \varepsilon_3$ with $\varepsilon_3$ small depending only on $||\chi_5||_{L^\infty(\mathbb{T}^{d+1})}$,  it is easy to see that
$$||\nabla {u}_\varepsilon||_{L^\infty_TL^2_x}+||\nabla^2 {u}_\varepsilon||_{L^2(\Omega_T)}\leq C\varepsilon^{-1}\left(||f||_{L^2(\Omega_T)}+||g||_{H^1(\Omega)}\right),$$
which implies the desired estimate $(6.14)$ since this estimate is due to Lemma 4.1 for $1>\varepsilon>\varepsilon_3$.
\end{proof}

Now we are ready to obtain the homogenization equation.\\

\noindent\textbf{Proof of Theorem 3.6: the homogenization equation}.

Using the notions in Lemma 6.1, and recalling that
\begin{equation}\tilde{\chi_5}(y,\tau)=\chi_5(y,\tau)-\mathcal{M}_\tau(\chi_5)(y),\end{equation}
 we have
\begin{equation}\begin{aligned}
\varepsilon^{1-k}W^\varepsilon u_\varepsilon&=\varepsilon \partial_t \tilde{\chi_5}^\varepsilon \cdot u_\varepsilon=\varepsilon\partial_t (\tilde{\chi_5}^\varepsilon u_\varepsilon)-\varepsilon\tilde{\chi_5}^\varepsilon \partial_t u_\varepsilon\\
&=\varepsilon\partial_t (\tilde{\chi_5}^\varepsilon u_\varepsilon)-\varepsilon\tilde{\chi_5}^\varepsilon\left(\Delta u_\varepsilon+\varepsilon^{1-k}W^\varepsilon u_\varepsilon+f\right)\\
&=\varepsilon\partial_t (\chi_5^\varepsilon u_\varepsilon)-\varepsilon\Delta(\tilde{\chi_5}^\varepsilon u_\varepsilon)+2\nabla_y \tilde{\chi_5}^\varepsilon \nabla u_\varepsilon\\
&\quad+{\varepsilon}^{-1}\Delta_y\tilde{\chi_5}^\varepsilon \cdot u_\varepsilon-\varepsilon\tilde{\chi_5}^\varepsilon f-\varepsilon^{2-k}\tilde{\chi_5}^\varepsilon W^\varepsilon u_\varepsilon.
\end{aligned}\end{equation}
To proceed, let
\begin{equation}
\chi_{7}=:\int_0^\tau (\tilde{\chi_5}W)(y,s)ds,
\end{equation}which is meaningful due to $\mathcal{M}_\tau(\tilde{\chi_5}W)=\mathcal{M}_\tau({\chi_5}W)-\mathcal{M}_\tau(W)\mathcal{M}_\tau({\chi_5})=0$. Then
the last term in the inequality $(6.17)$ equals to
\begin{equation}\begin{aligned}
-\varepsilon^{2}\partial_t\chi_{7}^\varepsilon\cdot u_\varepsilon
=&-\varepsilon^{2}\partial_t(\chi_{7}^\varepsilon u_\varepsilon)+\varepsilon^{2}\chi_{7}^\varepsilon \partial_t u_\varepsilon\\
=&-\varepsilon^{2}\partial_t(\chi_{7}^\varepsilon u_\varepsilon)+\varepsilon^{2}\chi_{7}^\varepsilon\left(\Delta u_\varepsilon+\varepsilon^{1-k}W^\varepsilon u_\varepsilon+f\right)\\
=&-\varepsilon^{2}\partial_t (\chi_{7}^\varepsilon u_\varepsilon)+\varepsilon^{2}\Delta_x (\chi_{7}^\varepsilon u_\varepsilon)-\varepsilon^{2}\Delta_x (\chi_{7}^\varepsilon u_\varepsilon)\\
&+\varepsilon^{2}\operatorname{div}(\chi_{7}^\varepsilon \nabla u_\varepsilon)-\varepsilon\nabla_y \chi_{7}^\varepsilon \nabla u_\varepsilon+\varepsilon^{2}\chi_{7}^\varepsilon f+\varepsilon^{3-k}\chi_{7}^\varepsilon W^\varepsilon u_\varepsilon.
\end{aligned}\end{equation}
Moreover, due to $\mathcal{M}_\tau(\tilde{\chi_5})=0$, we let
\begin{equation}\chi_{4}=\int_0^\tau \tilde{\chi_5}(y,s)ds.\end{equation}
Then a direct computation shows that
\begin{equation}\begin{aligned}
&{\varepsilon}^{-1}\Delta_y\tilde{\chi_5}^\varepsilon\cdot u_\varepsilon=\varepsilon^{k-1}\partial_t ( \Delta_y\chi_4^\varepsilon\cdot u_\varepsilon)-\varepsilon^{k-1}  \Delta_y\chi_4^\varepsilon \partial_t u_\varepsilon\\
=&\varepsilon^{k-1}\partial_t ( \Delta_y\chi_4^\varepsilon\cdot u_\varepsilon)-\varepsilon^{k-1}  \Delta_y\chi_4^\varepsilon\left(\Delta u_\varepsilon+\varepsilon^{1-k}W^\varepsilon u_\varepsilon+f\right)\\
=&\varepsilon^{k-1}\partial_t ( \Delta_y\chi_4^\varepsilon\cdot u_\varepsilon)-\varepsilon^{k-1}\Delta ( \Delta_y\chi_4^\varepsilon\cdot u_\varepsilon)+\varepsilon^{k-1}\Delta ( \Delta_y\chi_4^\varepsilon \cdot u_\varepsilon)\\
&-\varepsilon^{k-1}  \Delta_y\chi_4^\varepsilon \Delta u_\varepsilon-\Delta_y\chi_4^\varepsilon\cdot W^\varepsilon u_\varepsilon-\varepsilon^{k-1}  \Delta_y\chi_4^\varepsilon\cdot f.
\end{aligned}\end{equation}
Therefore, by setting $\widehat{u_\varepsilon}=:(1-\varepsilon\tilde{\chi_5}^\varepsilon+\varepsilon^
{2}\chi_{7}^\varepsilon-\varepsilon^{k-1}  \Delta_y\chi_4^\varepsilon)u_\varepsilon$, $\widehat{u_\varepsilon}$ satisfies the following parabolic equation:

\begin{equation}\left\{\begin{aligned}
\partial_t \widehat{u_\varepsilon}-\Delta \widehat{u_\varepsilon} &=f_{61}(u_\varepsilon), &\text{in }\Omega\times (0,T),\\
\widehat{u_\varepsilon}&=0, &\text{on }\partial\Omega\times (0,T),\\
\widehat{u_\varepsilon}&=g+\varepsilon\mathcal{M}_\tau({\chi_5})(x/\varepsilon)g, &\text{on }\Omega\times \{t=0\}.
\end{aligned}\right.\end{equation}
with

\begin{equation}\begin{aligned}
f_{61}(u_\varepsilon)=&(1-\varepsilon\tilde{\chi_5}^\varepsilon+\varepsilon^
{2}\chi_{7}^\varepsilon-\varepsilon^{k-1}  \Delta_y\chi_4^\varepsilon)f-\varepsilon^{2}\Delta_x (\chi_{7}^\varepsilon u_\varepsilon)
+\varepsilon^{2}\operatorname{div}(\chi_{7}^\varepsilon \nabla u_\varepsilon)\\
&-\varepsilon\nabla_y \chi_{7}^\varepsilon \nabla u_\varepsilon
+\varepsilon^{3-k}\chi_{7}^\varepsilon W^\varepsilon u_\varepsilon+2\nabla_y \tilde{\chi_5}^\varepsilon \nabla u_\varepsilon
+\varepsilon^{k-1}\Delta ( \Delta_y\chi_4^\varepsilon \cdot u_\varepsilon)\\
&\quad-\varepsilon^{k-1}  \Delta_y\chi_4^\varepsilon \Delta u_\varepsilon
-\Delta_y\chi_4^\varepsilon\cdot W^\varepsilon u_\varepsilon.
\end{aligned}\end{equation}

According to $\tilde{\chi_5},\chi_{7},\nabla_y\tilde{\chi_5},\nabla_y\chi_{7},\Delta_y\chi_4, \nabla_y^3\chi_4\in L^\infty(\mathbb{T}^{d+1}),$ and Lemma 6.1, it is easy to see that
\begin{equation*}\begin{aligned}
\widehat{u_\varepsilon}\in L^2(0,T;H^1_0(\Omega)),\ \partial_t\widehat{u_\varepsilon}\in L^2(0,T;H^{-1}(\Omega)).
\end{aligned}\end{equation*}

To proceed, firstly, since $\{u_\varepsilon\}$ is uniformly bounded in $L^2(0,T;H^1_0(\Omega))$, then up to a subsequence we still relabel by $\varepsilon$, there exists $u_0\in L^2(0,T;H^1_0(\Omega))$ such that $u_\varepsilon\rightharpoonup u_0$ weakly in $L^2(0,T;H^1_0(\Omega))$. Next, by Lemma 4.3 and the estimates above, we know that up to a subsequence we still relabel by $\varepsilon$, there exists $\tilde{u}_0\in L^2(\Omega_T)$ such that
$\widehat{u_\varepsilon}\rightarrow \tilde{u}_0$ strongly in $L^2(\Omega_T)$. Consequently, $u_0$ must equal to $\tilde{u}_0$ since $(-\varepsilon\tilde{\chi_5}^\varepsilon+\varepsilon^
{2}\chi_{7}^\varepsilon-\varepsilon^{k-1}  \Delta_y\chi_4^\varepsilon)u_\varepsilon\rightarrow 0$ strongly in $L^2(\Omega_T)$.

In conclusion,
\begin{equation}u_\varepsilon\rightharpoonup u_0\text{ weakly in }L^2(0,T;H^1_0(\Omega))\text{ and } u_\varepsilon\rightarrow u_0\text{  strongly in }L^2(\Omega_T).\end{equation}

To see the effective equation satisfied by $u_0$, we need the weak limit of $\nabla_y \tilde{\chi_5}^\varepsilon \nabla u_\varepsilon$ in $\mathcal{D}'(\Omega_T)$. A direct computation shows that
\begin{equation*}\begin{aligned}
&2\nabla_y \tilde{\chi_5}^\varepsilon \nabla u_\varepsilon=2\varepsilon^k\partial_t\nabla_y\chi_{4}^\varepsilon\cdot \nabla u_\varepsilon=2\varepsilon^k\partial_t(\nabla_y\chi_{4}^\varepsilon \nabla u_\varepsilon)-2\varepsilon^k\nabla_y\chi_{4}^\varepsilon \partial_t\nabla u_\varepsilon\\
=&2\varepsilon^k\partial_t(\nabla_y\chi_{4}^\varepsilon \nabla u_\varepsilon)-2\varepsilon^k\nabla_y\chi_{4}^\varepsilon\cdot\left(\nabla \Delta u_\varepsilon+\varepsilon^{-k}\nabla_yW^\varepsilon u_\varepsilon+\varepsilon^{1-k}W^\varepsilon \nabla u_\varepsilon+\nabla f\right)\\
=&2\varepsilon^k\partial_t(\nabla_y\chi_{4}^\varepsilon \nabla
 u_\varepsilon)-2\nabla_y\chi_{4}^\varepsilon\nabla_yW^\varepsilon\cdot
 u_\varepsilon-2\varepsilon^k\operatorname{div}(\nabla_y\chi_{4}^\varepsilon\Delta u_\varepsilon)+
2\varepsilon^{k-1}\Delta_y\chi_{4}^\varepsilon\Delta u_\varepsilon\\
&\quad-2\varepsilon\nabla_y\chi_{4}^\varepsilon\cdot W^\varepsilon\nabla
u_\varepsilon-2\varepsilon^k\operatorname{div}(f\nabla_y\chi_{4}^\varepsilon)+
2\varepsilon^{k-1}f\Delta_y\chi_{4}^\varepsilon.
\end{aligned}\end{equation*}

Therefore, we know that
\begin{equation}\begin{aligned}
f_{61}(u_\varepsilon)=&(1-\varepsilon\tilde{\chi_5}^\varepsilon+\varepsilon^
{2}\chi_{7}^\varepsilon+\varepsilon^{k-1}  \Delta_y\chi_4^\varepsilon)f-\varepsilon^{2}\Delta_x (\chi_{7}^\varepsilon u_\varepsilon)
+\varepsilon^{2}\operatorname{div}(\chi_{7}^\varepsilon \nabla u_\varepsilon)\\
&-\varepsilon\nabla_y \chi_{7}^\varepsilon \nabla u_\varepsilon
+\varepsilon^{3-k}\chi_{7}^\varepsilon W^\varepsilon u_\varepsilon
+\varepsilon^{k-1}\Delta ( \Delta_y\chi_4^\varepsilon \cdot u_\varepsilon)
-\varepsilon^{k-1}  \Delta_y\chi_4^\varepsilon \Delta u_\varepsilon\\
&\quad-2\varepsilon\nabla_y\chi_{4}^\varepsilon\cdot W^\varepsilon\nabla
u_\varepsilon+2\varepsilon^k\operatorname{div}(f\nabla_y\chi_{4}^\varepsilon)+
2\varepsilon^k\partial_t(\nabla_y\chi_{4}^\varepsilon \nabla
 u_\varepsilon)\\
 &\quad\quad-2\varepsilon^k\operatorname{div}(\nabla_y\chi_{4}^\varepsilon\Delta u_\varepsilon)
 +2\varepsilon^{k-1}\Delta_y\chi_{4}^\varepsilon\Delta u_\varepsilon-\Delta_y\chi_4^\varepsilon\cdot W^\varepsilon u_\varepsilon\\
 &\quad\quad\quad-2\nabla_y\chi_{4}^\varepsilon\nabla_yW^\varepsilon\cdot
 u_\varepsilon.
\end{aligned}\end{equation}

In view of Lemma 6.2, we have $\varepsilon^{k-1}\left(||\nabla {u}_\varepsilon||_{L^\infty_TL^2_x}+||\Delta u_\varepsilon||_{L^2(\Omega_T)}\right)\leq C\varepsilon^{k-2}\rightarrow 0$ for $2<k\leq 3$. Moreover, integration by parts yields that
\begin{equation}-\mathcal{M}( W\Delta_y\chi_4)=\mathcal{M}(\nabla_y W\nabla_y\chi_4).\end{equation}

 Consequently, with the convergence $(6.24)$ at hand, we can easily deduce from $(6.25)$ that $u_0$ satisfies the following homogenized equation for $2<k<3$:

\begin{equation}\left\{\begin{aligned}
\partial_t u_0-\Delta u_0+\mathcal{M}(\nabla_y\chi_{4}\nabla_yW) u_0&=f, \quad\quad\text{in }\Omega\times (0,T),\\
u_0&=0, \quad\quad\text{on }\partial\Omega\times (0,T),\\
u_0&=g, \quad\quad\text{on }\Omega\times \{t=0\}.
\end{aligned}\right.\end{equation}

Moreover, for $k=3$,
$u_0$ satisfies the following homogenized equation:

\begin{equation}\left\{\begin{aligned}
\partial_t u_0-\Delta u_0+\mathcal{M}(\nabla_y\chi_{4}\nabla_yW) u_0&=f, \quad\quad\text{in }\Omega\times (0,T),\\
u_0&=0, \quad\quad\text{on }\partial\Omega\times (0,T),\\
u_0&=g, \quad\quad\text{on }\Omega\times \{t=0\},
\end{aligned}\right.\end{equation}
since a direct computation shows that, for every $y\in \mathbb{T}$,
\begin{equation}\begin{aligned}
\int_{\mathbb{T}}(\chi_{7} W)(\tau)d\tau=&\int_{\mathbb{T}} W(\tau)\int_0^\tau \tilde{\chi}_5(t)W(t)dtd\tau\\
=&\int_{\mathbb{T}} W(\tau)\int_0^\tau W(t)\chi_5(t)dtd\tau-\mathcal{M}_\tau(\chi_5)\int_{\mathbb{T}} W(\tau)\int_0^\tau W(t)dtd\tau\\
=&\int_{\mathbb{T}} W(\tau)\int_0^\tau W(t)\int_0^s W(s)dsdtd\tau-\mathcal{M}_\tau(\chi_5)\int_{\mathbb{T}} W(\tau)\int_0^\tau W(t)dtd\tau\\
=&\frac 16 \left(\int_{\mathbb{T}}W(\tau)d\tau\right)^3-\frac 12\mathcal{M}_\tau(\chi_5)\left(\int_{\mathbb{T}}W(\tau)d\tau\right)^2=0.
\end{aligned}\end{equation}\qed\\

After determining the effective equation, we are in a position to bound the error.\\

\noindent\textbf{Proof of the second part of Theorem 3.6: the convergence rates.}\\

Defining
\begin{equation}w_\varepsilon=(1-\varepsilon\tilde{\chi_5}^\varepsilon+\varepsilon^
{2}\chi_{7}^\varepsilon-\varepsilon^{k-1}  \Delta_y\chi_4^\varepsilon)v_\varepsilon -v_0,\end{equation} and
going back to $(6.23)$ since we do not know the boundary information for $\nabla v_\varepsilon$, $w_\varepsilon$ satisfies

\begin{equation}\left\{\begin{aligned}
\partial_t w_\varepsilon-\Delta w_\varepsilon+Nw_\varepsilon &=f_{62}(v_\varepsilon),& &\text{in }\Omega\times (0,T),\\
w_\varepsilon&=0,& &\text{on }\partial\Omega\times (0,T),\\
w_\varepsilon&=\varepsilon\mathcal{M}_\tau({\chi_5})(x/\varepsilon)g,&\quad &\text{on }\Omega\times \{t=0\},
\end{aligned}\right.\end{equation}
with
\begin{equation*}\begin{aligned}
f_{62}(v_\varepsilon)=(-\varepsilon\tilde{\chi_5}^\varepsilon+\varepsilon^
{2}\chi_{7}^\varepsilon-\varepsilon^{k-1}  \Delta_y\chi_4^\varepsilon)\tilde{f}-\varepsilon^{2}\Delta_x (\chi_{7}^\varepsilon v_\varepsilon)
+\varepsilon^{2}\operatorname{div}(\chi_{7}^\varepsilon \nabla v_\varepsilon)\\
-\varepsilon\nabla_y \chi_{7}^\varepsilon \nabla v_\varepsilon
+\varepsilon^{3-k}\chi_{7}^\varepsilon W^\varepsilon v_\varepsilon+2\nabla_y \tilde{\chi_5}^\varepsilon \nabla v_\varepsilon
+\varepsilon^{k-1}\Delta ( \Delta_y\chi_4^\varepsilon \cdot u_\varepsilon)\\
\quad-\varepsilon^{k-1}  \Delta_y\chi_4^\varepsilon \Delta v_\varepsilon
-\Delta_y\chi_4^\varepsilon\cdot W^\varepsilon v_\varepsilon+\mathcal{M}(\nabla_y\chi_{4}\nabla_yW) v_0.
\end{aligned}\end{equation*}
To bound $w_\varepsilon$, we rewrite it as $w_\varepsilon=w_{5,\varepsilon}+w_{6,\varepsilon}$ with $w_{5,\varepsilon}$ defined as
\begin{equation}\left\{\begin{aligned}
\partial_t w_{5,\varepsilon}-\Delta w_{5,\varepsilon}+Nw_{5,\varepsilon} &=f_{63}(v_\varepsilon),&\quad &\text{in }\Omega\times (0,T),\\
w_{5,\varepsilon}&=0,& &\text{on }\partial\Omega\times (0,T),\\
w_{5,\varepsilon}&=0,& &\text{on }\Omega\times \{t=0\},
\end{aligned}\right.\end{equation}
where
$$\begin{aligned}f_{63}(v_\varepsilon)=:(-\varepsilon\tilde{\chi_5}^\varepsilon+\varepsilon^
{2}\chi_{7}^\varepsilon-\varepsilon^{k-1}  \Delta_y\chi_4^\varepsilon)\tilde{f}-\varepsilon^{2}\Delta_x (\chi_{7}^\varepsilon v_\varepsilon)
+\varepsilon^{2}\operatorname{div}(\chi_{7}^\varepsilon \nabla v_\varepsilon)\\
-\varepsilon\nabla_y \chi_{7}^\varepsilon \nabla v_\varepsilon
+\varepsilon^{k-1}\Delta ( \Delta_y\chi_4^\varepsilon \cdot u_\varepsilon)-\varepsilon^{k-1}  \Delta_y\chi_4^\varepsilon \Delta v_\varepsilon.\end{aligned}$$

Due to $\varepsilon^{k-1}\Delta ( \Delta_y\chi_4^\varepsilon \cdot u_\varepsilon)=\varepsilon^{k-1}\operatorname{div} ( \Delta_y\chi_4^\varepsilon \cdot \nabla u_\varepsilon)+\varepsilon^{k-2}\operatorname{div} (\nabla_y \Delta_y\chi_4^\varepsilon \cdot u_\varepsilon)$
and $\nabla_y^3 W\in L^\infty(\mathbb{T}^{d+1})$,
it is easy to see that
\begin{equation}\begin{aligned}
||w_{5,\varepsilon}||_{L^\infty_TL^2_x}&+||\nabla w_{5,\varepsilon}||_{L^2(\Omega_T)}+|| w_{5,\varepsilon}||_{L^2(\Omega_T)}\\
&\leq C\varepsilon^{k-2}\left(||f||_{L^2(\Omega_T)}+||g||_{H^1(\Omega)}\right).
\end{aligned}\end{equation}
To bound $w_{6,\varepsilon}$, we first note that $w_{6,\varepsilon}\in W^{2,1}_2(\Omega_T)$ satisfies
\begin{equation}\left\{\begin{aligned}
\partial_t w_{6,\varepsilon}-\Delta w_{6,\varepsilon}+Nw_{6,\varepsilon} &=f_{64}(v_\varepsilon),& &\text{in }\Omega\times (0,T),\\
w_{6,\varepsilon}&=0,& &\text{on }\partial\Omega\times (0,T),\\
w_{6,\varepsilon}&=\varepsilon\mathcal{M}_\tau({\chi_5})(x/\varepsilon)g,&\quad &\text{on }\Omega\times \{t=0\},
\end{aligned}\right.\end{equation}
with
$$f_{64}(v_\varepsilon)=:2\nabla_y \tilde{\chi_5}^\varepsilon \nabla v_\varepsilon
-\Delta_y\chi_4^\varepsilon\cdot W^\varepsilon v_\varepsilon+\mathcal{M}(\nabla_y\chi_{4}\nabla_yW) v_0+\varepsilon^{3-k}\chi_{7}^\varepsilon W^\varepsilon v_\varepsilon.$$
We first note that
\begin{equation}\begin{aligned}
\int_{\Omega_T}2\nabla_y \tilde{\chi_5}^\varepsilon \nabla v_\varepsilon\cdot w_{6,\varepsilon}
=&\int_{\Omega_T}2\nabla_y \tilde{\chi_5}^\varepsilon \nabla v_0\cdot w_{6,\varepsilon}+
\int_{\Omega_T}2\nabla_y \tilde{\chi_5}^\varepsilon \nabla w_{6,\varepsilon}\cdot w_{6,\varepsilon}\\
&+\int_{\Omega_T}2\nabla_y \tilde{\chi_5}^\varepsilon \nabla\left[(-\varepsilon^
{2}\chi_{7}^\varepsilon+\varepsilon^{k-1}  \Delta_y\chi_4^\varepsilon)v_\varepsilon- w_{5,\varepsilon}\right]\cdot w_{6,\varepsilon}\\
&+\int_{\Omega_T}2|\nabla_y \tilde{\chi_5}^\varepsilon|^2 v_\varepsilon w_{6,\varepsilon}+
2\int_{\Omega_T}\varepsilon\tilde{\chi_5}^\varepsilon\nabla_y \tilde{\chi_5}^\varepsilon\nabla v_\varepsilon \cdot w_{6,\varepsilon}\\
=&:\sum_{i=3}^7M_i.
\end{aligned}\end{equation}
Integration by parts yields that
\begin{equation}
|M_3|\leq C\varepsilon ||\Delta v_0||_{L^2(\Omega_T)}||w_{6,\varepsilon}||_{L^2(\Omega_T)}+C\varepsilon ||\nabla v_0||_{L^2(\Omega_T)}||\nabla w_{6,\varepsilon}||_{L^2(\Omega_T)}.
\end{equation}
And
9t is easy to see that
\begin{equation}
|M_4|\leq \frac{1}{16}||\nabla w_{6,\varepsilon}||_{L^2(\Omega_T)}^2+C||w_{6,\varepsilon}||_{L^2(\Omega_T)}^2.
\end{equation}
Moreover, a direct computation after noting $(6.33)$ yields that
\begin{equation}
|M_5|\leq C\varepsilon^{k-2}\left(||f||_{L^2(\Omega_T)}+||g||_{H^1(\Omega)}\right)
||w_{6,\varepsilon}||_{L^2(0,T;H^1(\Omega))}.
\end{equation}
Before we proceed, according to $(6.16)$ and $(6.20)$, a direct computation yields
\begin{equation}\begin{aligned}
\int_{\mathbb{T}^{d+1}}\nabla_y\chi_4 \nabla_y W=&\int_{\mathbb{T}^{d+1}}\nabla_y\chi_4 \nabla_y \partial_\tau\tilde{\chi}_5=-\int_{\mathbb{T}^{d+1}}\nabla_y\partial_\tau\chi_4 \nabla_y\tilde{\chi}_5\\
=&-\int_{\mathbb{T}^{d+1}} \nabla_y\tilde{\chi}_5\nabla_y\tilde{\chi}_5,
\end{aligned}\end{equation}
which, after combining $(6.26)$, implies that
$$2\mathcal{M}(|\nabla_y \tilde{\chi}_5|^2)+\mathcal{M}(\nabla_y\chi_4 \nabla_y W)-\mathcal{M}(W\Delta_y\chi_4 )=0.$$

Therefore, to complete this proof, we need only to solve a cell problem with source term $2|\nabla_y \tilde{\chi}_5|^2-\Delta_y\chi_4\cdot W
+\mathcal{M}(\nabla_y\chi_{4}\nabla_yW)$ in $\mathbb{T}^{d+1}$. More exactly, we introduce $\chi_{6-2}(y,\tau)$ solving the following cell problem:
\begin{equation}\left\{\begin{aligned}
\partial_\tau {\chi_{6-2}} -\Delta_y {\chi_{6-2}}=2|\nabla_y \tilde{\chi}_5|^2-\Delta_y\chi_4\cdot W
+\mathcal{M}(\nabla_y\chi_{5}\nabla_yW)\text{ in }\mathbb{T}^{d+1},\\
 {\chi_{6-2}}\text{ is 1-periodic in }(y,\tau)\text{ with }\mathcal{M}({\chi_{6-2}})=0.
\end{aligned}\right.\end{equation}
Then $${\chi_{6-2}}\in W^{2,1}_p(\mathbb{T}^{d+1}),\text{ for any }1<p<\infty,$$
and by embedding,
$$ {\chi_{6-2}} \in C^{1+\alpha,\frac{1+\alpha}2}(\mathbb{T}^{d+1}), \text{ for any }0<\alpha<1.$$
Then a direct computation shows that
\begin{equation}\begin{aligned}
\int_{\Omega_T}&\left(2|\nabla_y \tilde{\chi_5}^\varepsilon|^2 v_\varepsilon-\Delta_y\chi_4^\varepsilon\cdot W^\varepsilon v_\varepsilon+\mathcal{M}(\nabla_y\chi_{4}\nabla_yW) v_0\right) w_{6,\varepsilon}\\
&=\int_{\Omega_T}\left(2|\nabla_y \tilde{\chi_5}^\varepsilon|^2 -\Delta_y\chi_4^\varepsilon\cdot W^\varepsilon +\mathcal{M}(\nabla_y\chi_{4}\nabla_yW) \right)v_0 w_{6,\varepsilon}\\
&\quad\quad+\int_{\Omega_T}\left(2|\nabla_y \tilde{\chi_5}^\varepsilon|^2 -\Delta_y\chi_4^\varepsilon\cdot W^\varepsilon \right)(v_\varepsilon-v_0) w_{6,\varepsilon}\\
&=\int_{\Omega_T}\left(\varepsilon^k\partial_t\chi_{6-2}^\varepsilon-\varepsilon^2 \Delta_x\chi_{6-2}^\varepsilon\right)v_0 w_{6,\varepsilon}+M_{9}\\
&=:M_8+M_9.
\end{aligned}\end{equation}

\noindent
It is then easy to see that
\begin{equation}\begin{aligned}
|M_{9}|\leq& C||w_{6,\varepsilon}||_{L^2(\Omega_T)}^2+C\varepsilon||v_\varepsilon||_{L^2(\Omega_T)}||w_{6,\varepsilon}||_{L^2(\Omega_T)}
+C||w_{5,\varepsilon}||_{L^2(\Omega_T)}||w_{6,\varepsilon}||_{L^2(\Omega_T)}\\
\leq& C||w_{6,\varepsilon}||_{L^2(\Omega_T)}^2+C\varepsilon^{2k-4}
\left(||f||_{L^2(\Omega_T)}^2+||g||_{H^1(\Omega)}^2\right).
\end{aligned}\end{equation}

\noindent
And a direct computation yields that
\begin{equation}\begin{aligned}
|M_8|\leq& C\varepsilon^k||v_0||_{L^\infty_TL^2_x}||w_{6,\varepsilon}||_{L^\infty_TL^2_x}
+C\varepsilon^k||\partial_tv_0||_{L^2(\Omega_T)}||w_{6,\varepsilon}||_{L^2(\Omega_T)}\\
&+C\varepsilon^k||v_0||_{L^2(\Omega_T)}||\partial_t w_{6,\varepsilon}||_{L^2(\Omega_T)}+C\varepsilon|| v_0||_{L^2(0,T;H^1(\Omega))}|| w_{6,\varepsilon}||_{L^2(0,T;H^1(\Omega))}\\
\leq& \frac1{16}||w_{6,\varepsilon}||_{L^\infty_TL^2_x}^2+\frac1{16}||\nabla w_{6,\varepsilon}||_{L^2(\Omega_T)}^2+|| w_{6,\varepsilon}||_{L^2(\Omega_T)}^2\\
&+C\varepsilon^{2}\left(||f||_{L^2(\Omega_T)}^2+||g||_{H^1(\Omega)}^2\right).
\end{aligned}\end{equation}

Moreover, we need to estimate the following term
\begin{equation}\int_{\Omega_T}\varepsilon^{3-k}\left(\chi_{7}^\varepsilon W^\varepsilon v_\varepsilon-\mathcal{M}(\chi_{7} W)v_0\right)w_{6,\varepsilon},\end{equation}
after noting that $\mathcal{M}_\tau (\chi_{7} W)=0$ for every $y\in\mathbb{T}^d$ (see $(6.29)$),
which can be estimated as the same way of the previous content and would provide $O(\varepsilon)$ convergence rates for us.

Consequently, multiplying the equation $(6.34)$ by $w_{6,\varepsilon}$ after combining $(6.33)$ yields that:
\begin{equation}
||w_\varepsilon||_{L^2(\Omega_T)}+||w_\varepsilon||_{L^\infty_TL^2_x}+||\nabla w_{\varepsilon}||_{L^2(\Omega_T)}\leq C\varepsilon^{k-2}\left(||f||_{L^2(\Omega_T)}+||g||_{H^1(\Omega)}\right).
\end{equation}

\noindent
Thus, we complete the proof of Theorem 3.6.\qed\\

\begin{center}\textbf{\large{Acknowlwdgement}}\end{center}

\normalem\bibliographystyle{plain}{}

%\bibliography{bib}
\end{document}